\PassOptionsToPackage{dvipsnames}{xcolor} 
\documentclass[11pt,letter]{article}
\usepackage{amssymb,amsmath,amsthm,multirow,color,xcolor,cancel,bbm,tikz,mathrsfs,hyperref,authblk,stmaryrd,dsfont,authblk}
\usepackage[textheight=24cm,textwidth=16cm]{geometry}
\hypersetup{
    colorlinks=true,
    linkcolor=blue,
    citecolor=cyan,
    filecolor=magenta,      
    urlcolor=green,
    pdftitle={Overleaf Example},
    pdfpagemode=FullScreen,
    }

\newtheorem{definition}{Definition}
\newtheorem{proposition}{Proposition\setcounter{claimcounter}{0}}
\newtheorem{lemma}{Lemma\setcounter{claimcounter}{0}}
\newtheorem{theorem}{Theorem\setcounter{claimcounter}{0}}
\newtheorem{corollary}{Corollary}

\newtheorem{example}{Example}

\newcounter{claimcounter}
\newtheorem*{claim*}{{\it Claim}}

\newcommand{\EM}[1]{{\it\textcolor{Maroon}{#1}}}
\newcommand{\EMM}[1]{{\textcolor{Maroon}{#1}}}

\def\lcm{\mathrm{lcm}}

\def\Mult{\mathrm{Mult}}

\def\Div{\mathrm{Div}}
\def\div{\mathrm{div}}
\def\sol{\mathrm{sol}}
\def\Sol{\mathrm{Sol}}
\def\red{\mathrm{red}}
\def\irred{\mathrm{irred}}
\def\fact{\mathrm{fact}}
\def\L{\mathcal{L}}

\def\F{\mathbb{F}}
\def\S{\mathbb{S}}
\def\N{\mathbb{N}}

\def\d{\oslash}
\def\dd{\boxslash}
\def\m{\otimes}

\title{Dividing sum of cycles in the semiring of functional digraphs}

\author[1]{Florian Bridoux}
\author[1]{Christophe Crespelle}
\author[2]{Thi Ha Duong Phan}
\author[1]{Adrien Richard}

\affil[1]{\small Universit\'e C\^ote d’Azur, CNRS, I3S, Sophia Antipolis, France 
\tt{\small\{florian.bridoux,christophe.crespelle,adrien.richard\}@univ-cotedazur.fr}}

\affil[2]{\small ICRTM - Institute of Mathematics, Vietnam Academy of Science and Technology, Vietnam
\tt{\small phanhaduong@math.ac.vn}}

\begin{document}

\maketitle

\begin{abstract}
Functional digraphs are unlabelled finite digraphs where each vertex has exactly one out-neighbor. They are isomorphic classes of  finite discrete-time dynamical systems. Endowed with the direct sum and product, functional digraphs form a semiring with an interesting multiplicative structure. For instance, we do not know if the following division problem can be solved in polynomial time: given two functional digraphs $A$ and $B$, does $A$ divide $B$? That $A$ divides $B$ means that there exists a functional digraph $X$ such that $AX$ is isomorphic to $B$, and many such $X$ can exist. We can thus ask for the number of solutions $X$. In this paper, we focus on the case where $B$ is a sum of cycles (a disjoint union of cycles, corresponding to the limit behavior of finite discrete-time dynamical systems). There is then a naïve sub-exponential algorithm to compute the non-isomorphic solutions $X$, and our main result is an improvement of this algorithm which has the property to be polynomial when $A$ is fixed. It uses a divide-and-conquer technique that should be useful for further developments on the division problem.   

\end{abstract}

%%%%%%%%%%%%%%%%%%%%%%%%%%%%%%%%%%%%%%%%%%%%%%%%%%%%%%%%%%%%%%%%%%%%%%%%%%%%%%%%%%%%%%%%%
%%%%%%%%%%%%%%%%%%%%%%%%%%%%%%%%%%%%%%%%%%%%%%%%%%%%%%%%%%%%%%%%%%%%%%%%%%%%%%%%%%%%%%%%%
\section{Introduction}
%%%%%%%%%%%%%%%%%%%%%%%%%%%%%%%%%%%%%%%%%%%%%%%%%%%%%%%%%%%%%%%%%%%%%%%%%%%%%%%%%%%%%%%%%
%%%%%%%%%%%%%%%%%%%%%%%%%%%%%%%%%%%%%%%%%%%%%%%%%%%%%%%%%%%%%%%%%%%%%%%%%%%%%%%%%%%%%%%%%

A deterministic, finite, discrete-time dynamical system is a function from a finite set (of states or configurations) to itself. Equivalently, this is a \EM{functional digraph}, that is, a finite directed graph where each vertex has a unique out-neighbor. In addition to being ubiquitous objects in discrete mathematics, such systems have many real-life applications: they are usual models for the dynamics of gene networks \cite{K69,T73,TK01,J02}, neural networks \cite{MP43,H82,G85}, reaction systems \cite{ehrenfeucht2007reaction}, social interactions \cite{PS83, GT83} and more \cite{TA90,GM90}. 

\medskip
A functional digraph $A$ basically contains two parts: the \EM{cyclic part}, which is the collection of the cycles of $A$ (which are vertex  disjoint), and the \EM{transient part}, which is obtained by deleting arcs in cycles, and which is a disjoint union of out-trees; see Figure \ref{fig:cyclic_transient}. From a dynamical point of view, the cyclic part is fundamental since it describes the asymptotic behavior of the system, and this is the part we focus on in this paper. 

\medskip
Furthermore, we consider functional digraphs up to isomorphism; two functional digraphs are isomorphic if there exists a one-to-one correspondence between their vertices that preserves the arcs (they are identical up to a relabeling of the vertices). This is motivated by the fact that many studied dynamical properties are invariant by isomorphism: number of fixed points, periodic points, limit cycles, lengths of limits cycles and so on; see e.g. \cite{FO89} in the context of random functional digraphs, and \cite{R19,G20} in the context of automata networks. An isomorphism class then corresponds to an unlabelled functional digraph, and we write \EM{$A=B$} to mean that $A$ is isomorphic to $B$. Since a planar embedding of a functional digraph can be obtained in linear time, testing if $A=B$ can be done in linear time \cite{hopcroft1974linear}.

\begin{figure}
\[
\begin{array}{cc}
\begin{array}{c}
\begin{tikzpicture}
\node[outer sep=0,inner sep=1] (a) at (180:0.5){{\small$1$}};
\node[outer sep=0,inner sep=1] (b) at (0:0.5){{\small$2$}};
\node[outer sep=0,inner sep=1] (1) at (30:1.2){{\small$3$}};
\node[outer sep=0,inner sep=1] (2) at (50:1.6){{\small$4$}};
\node[outer sep=0,inner sep=1] (3) at (30:1.9){{\small$5$}};
\node[outer sep=0,inner sep=1] (4) at (10:1.6){{\small$6$}};
\node[outer sep=0,inner sep=1] (5) at (165:1.2){{\small$7$}};
\node[outer sep=0,inner sep=1] (6) at (195:1.2){{\small$8$}};

\path[->,Red,thick]
(a) edge[bend left=60] (b)
(b) edge[bend left=60] (a)
;

\path[->,Green,thick]
(1) edge (b)
(2) edge (1)
(3) edge (1)
(4) edge (1)
(5) edge (a)
(6) edge (a)
;
\end{tikzpicture}
\end{array}
&
\begin{array}{c}
\begin{tikzpicture}
\node[outer sep=0,inner sep=1] (1) at (180:0.7){{\small$9$}};
\node[outer sep=0,inner sep=1] (2) at (60:0.7){{\small$10$}};
\node[outer sep=0,inner sep=1] (3) at (-60:0.7){{\small$11$}};
\node[outer sep=0,inner sep=1] (a) at (60:1.4){{\small$12$}};
\node[outer sep=0,inner sep=1] (b) at (70:2){{\small$13$}};
\node[outer sep=0,inner sep=1] (c) at (50:2){{\small$14$}};
\node[outer sep=0,inner sep=1] (d) at (160:1.3){{\small$15$}};

\path[->,Red,thick]
(1) edge[bend left=40] (2)
(2) edge[bend left=40] (3)
(3) edge[bend left=40] (1)
;

\path[->,Green,thick]
(a) edge (2)
(b) edge (a)
(c) edge (a)
(d) edge (1)
;

\end{tikzpicture}
\end{array}
\end{array}
\]
{\caption{\label{fig:cyclic_transient} A functional digraph, with cyclic part in red, and transient part in green.}}
\end{figure}

\medskip
There are two natural algebraic operations to obtain larger systems from smaller ones. Given two functional digraphs $A$ and $B$, the \EM{addition} \EM{$A+B$} is the disjoint union of $A$ and $B$, while the \EM{multiplication} \EM{$A\cdot B$} (or simply \EM{$AB$}) is the \EM{direct product} of $A$ and $B$: the vertex set of $AB$ is the Cartesian product of the vertex set of $A$ and the vertex set of $B$, and the out-neighbor of $(a,b)$ in $AB$ is $(a',b')$ where $a'$ is the out-neighbor of $a$ in $A$ and $b'$ is the out-neighbor of vertex $b$ in $B$; see Figure \ref{fig:product}. Hence $AB$ describes the parallel evolution of the dynamics described by $A$ and $B$. Endowed with these two operations, the set of functional digraphs forms a semiring \EM{$\F$}, introduced as such in~\cite{dennunzio2018polynomial}. The identity element for the addition is the empty functional digraph, while for the product it is the cycle of length one, denoted \EM{$C_1$}.

\begin{figure}
\[
\begin{array}{c}
\begin{tikzpicture}
\node[outer sep=1,inner sep=1] (a) at (180:0.5){\small{$a$}};
\node[outer sep=1,inner sep=1] (b) at (0:0.5){\small{$b$}};
\node[outer sep=1,inner sep=1] (c) at (140:1){\small{$c$}};
\path[->,thick]
(a) edge[Red,bend left=60] (b)
(b) edge[Red,bend left=60] (a)
(c) edge[Green] (a)
;
\end{tikzpicture}
\end{array}
\cdot
\begin{array}{c}
\begin{tikzpicture}
\node[outer sep=1,inner sep=1] (1) at (180:0.7){\small{$1$}};
\node[outer sep=1,inner sep=1] (2) at (60:0.7){\small{$2$}};
\node[outer sep=1,inner sep=1] (3) at (-60:0.7){\small{$3$}};
\node[outer sep=1,inner sep=1] (5) at (140:1.9){\small{$5$}};
\node[outer sep=1,inner sep=1] (4) at (150:1.2){\small{$4$}};
\path[->,thick]
(1) edge[Red,bend left=40] (2)
(2) edge[Red,bend left=40] (3)
(3) edge[Red,bend left=40] (1)
(4) edge[Green] (1)
(5) edge[Green] (4)
;
\end{tikzpicture}
\end{array}
=
\begin{array}{c}
\begin{tikzpicture}
\node[outer sep=1,inner sep=1.2] (a1) at (180:1){\small{$a$}\small{$1$}};
\node[outer sep=1,inner sep=1.2] (b2) at (120:1){\small{$b$}\small{$2$}};
\node[outer sep=1,inner sep=1.2] (a3) at (60:1){\small{$a$}\small{$3$}};
\node[outer sep=1,inner sep=1.2] (b1) at (0:1){\small{$b$}\small{$1$}};
\node[outer sep=1,inner sep=1.2] (a2) at (-60:1){\small{$a$}\small{$2$}};
\node[outer sep=1,inner sep=1.2] (b3) at (-120:1){\small{$b$}\small{$3$}};
\node[outer sep=1,inner sep=1.2] (b4) at (160:1.8){\small{$b$}\small{$4$}};
\node[outer sep=1,inner sep=1.2] (c3) at (180:1.8){\small{$c$}\small{$3$}};
\node[outer sep=1,inner sep=1.2] (c4) at (200:1.8){\small{$c$}\small{$4$}};
\node[outer sep=1,inner sep=1.2] (a5) at (150:2.6){\small{$a$}\small{$5$}};
\node[outer sep=1,inner sep=1.2] (c2) at (60:1.8){\small{$c$}\small{$2$}};
\node[outer sep=1,inner sep=1.2] (b5) at (20:2.2){\small{$b$}\small{$5$}};
\node[outer sep=1,inner sep=1.2] (a4) at (0:1.8){\small{$a$}\small{$4$}};
\node[outer sep=1,inner sep=1.2] (c5) at (-20:2.2){\small{$c$}\small{$5$}};
\node[outer sep=1,inner sep=1.2] (c1) at (-60:1.8){\small{$c$}\small{$1$}};

\path[->,thick]
(a1) edge[Red,bend left=20] (b2)
(b2) edge[Red,bend left=20] (a3)
(a3) edge[Red,bend left=20] (b1)
(b1) edge[Red,bend left=20] (a2)
(a2) edge[Red,bend left=20] (b3)
(b3) edge[Red,bend left=20] (a1)
(b4) edge[Green] (a1)
(c3) edge[Green] (a1)
(c4) edge[Green] (a1)
(a5) edge[Green] (b4)
(c2) edge[Green] (a3)
(a4) edge[Green] (b1)
(b5) edge[Green] (a4)
(c5) edge[Green] (a4)
(c1) edge[Green] (a2)
;
\end{tikzpicture}
\end{array}
\]
{\caption{\label{fig:product} Product of two functional digraphs.}}
\end{figure}

\medskip
The semiring $\F$ contains an isomorphic copy of $\N$ (by identifying each integer $n$ with the sum of $n$ copies of $C_1$). The set of sums of cycles \EM{$\S$}, which describes the cyclic parts of functional digraphs, is also a sub-semiring of $\F$. Because the multiplication of cycles involves the least common multiple operation,  the multiplicative structures of $\F$ and $\S$ have interesting properties, which differ from those of $\N$ (or the semiring of polynomials), and have been recently studied in \cite{DDFMP18,dennunzio2019solving,GMP20,riva2022factorisation,DFMR23,dore2024roots,dore2024decomposition,naquin2024factorisation,dennunzio2024note}, leading to many interesting open complexity problems.  

\medskip
To emphasize such differences, we need some definitions. Let us say that $X\in\F$ is \EM{irreducible} if, for all $A,B\in \F$, $X=AB$ implies $A=C_1$ or $B=C_1$. Let say that $A$ \EM{divides} $B$, \EM{$A\mid B$} in notation, if there is $X$ such that $AX=B$. Finally, let us say that $X$ is \EM{prime} if $X\neq C_1$ and, for all $A,B\in\F$, $X\mid AB$ implies $X\mid A$ or $X\mid B$. In $\N$, irreducibility and primeness are a same concept, but this is far from being true for $\F$: almost all functional digraphs are irreducible \cite{Dorigatti2017}, while Seifert \cite{seifert1971prime} proved, with technical arguments, that there is no prime functional digraph. Using only very simple arguments, we will prove that $\S$ share the same two properties. In practice, asking for irreducibility is very natural: given an observed dynamical system, does it correspond (up to isomorphism) to the parallel evolution of two smaller systems? Testing irreducibility is obviously in coNP, but this problem does not seem to have been studied in depth. 

\medskip
Another key point is that functional digraphs do not have a unique factorization into irreducibles (distinct from $C_1$). This is even true for sums of cycles. For instance, denoting \EM{$C_\ell$} the cycle of length $\ell$, the sum of cycles $C_2+C_2$ has two distinct irreducible factorizations, which are $C_2\cdot C_2$ and $C_2\cdot (C_1+C_1)$; see Figure \ref{fig:2C2}. Actually, we will prove that, for every $\epsilon>0$ and infinitely many $n$, there exists a sum of cycles with $n^2$ vertices which has at least 
\begin{equation}\label{eq:super_poly}
n^{2^{\frac{\ln n}{(1+\epsilon)\ln \ln n}}}
\end{equation}
irreducible factorizations, which is super-polynomial in $n$. Using a very rough argument, we will also show that every sum of cycles with $n$ vertices has at most $e^{O(\sqrt{n})}$ irreducible factorizations, which is sub-exponential, and we think that this upper-bound is far from being accurate. Here again, there is a lack of theoretical results concerning the complexity of finding an irreducible factorization. 

\begin{figure}
\[
\begin{array}{c}
\begin{tikzpicture}
\def\s{0.8}
\useasboundingbox (-0.1,-0.1) rectangle ({\s*3+0.1},0.2);
%\draw (-0.1,-0.1) rectangle ({\s*3+0.1},0.2);
\node[inner sep=1] (1) at ({\s*0},0){\small $1$};
\node[inner sep=1] (2) at ({\s*1},0){\small $2$};
\node  at ({\s*1.5},0){$\cdot$};
\node[inner sep=1] (a) at ({\s*2},0){\small $a$};
\node[inner sep=1] (b) at ({\s*3},0){\small $b$};
\path[thick,->,draw,Red]
(1) edge[bend left=60] (2)
(2) edge[bend left=60] (1)
(a) edge[bend left=60] (b)
(b) edge[bend left=60] (a)
;
\end{tikzpicture}
\end{array}
~=~
\begin{array}{c}
\begin{tikzpicture}
\def\s{0.8}
\useasboundingbox (-0.15,-0.1) rectangle ({\s*3.5+0.15},0.2);
%\draw (-0.15,-0.1) rectangle ({\s*3.5+0.15},0.2);
\node[inner sep=1] (1a) at ({\s*0},0){\small $1a$};
\node[inner sep=1] (2b) at ({\s*1},0){\small $2b$};
\node  at ({\s*1.75},0){\scriptsize $+$};
\node[inner sep=1] (1b) at ({\s*2.5},0){\small $1b$};
\node[inner sep=1] (2a) at ({\s*3.5},0){\small $2a$};
\path[thick,->,draw,Red]
(1a) edge[bend left=60] (2b)
(2b) edge[bend left=60] (1a)
(1b) edge[bend left=60] (2a)
(2a) edge[bend left=60] (1b)
;
\end{tikzpicture}
\end{array}
~=~
\begin{array}{c}
\begin{tikzpicture}
\def\s{0.8}
\useasboundingbox (-0.15,-0.1) rectangle ({\s*3.5+0.15},0.2);
%\draw (-0.15,-0.1) rectangle ({\s*3.5+0.15},0.2);
\node[inner sep=1] (1a) at ({\s*0},0){\small $1a$};
\node[inner sep=1] (2b) at ({\s*1},0){\small $2a$};
\node  at ({\s*1.75},0){\scriptsize $+$};
\node[inner sep=1] (1b) at ({\s*2.5},0){\small $1b$};
\node[inner sep=1] (2a) at ({\s*3.5},0){\small $2b$};
\path[thick,->,draw,Red]
(1a) edge[bend left=60] (2b)
(2b) edge[bend left=60] (1a)
(1b) edge[bend left=60] (2a)
(2a) edge[bend left=60] (1b)
;
\end{tikzpicture}
\end{array}
~=~
\begin{array}{c}
\begin{tikzpicture}
\def\s{0.8}
\useasboundingbox (-0.1,-0.2) rectangle ({\s*4.33+0.1},0.3);
%\draw (-0.1,-0.2) rectangle ({\s*4.33+0.1},0.3);
\node[inner sep=1] (1) at ({\s*0},0){\small $1$};
\node[inner sep=1] (2) at ({\s*1},0){\small $2$};
\node  at ({\s*1.33},0){$\cdot$};
\node  at ({\s*1.66},0){$\big($};
\node[inner sep=1] (a) at ({\s*2.35},0){\small $a$};
\node  at ({\s*3},0){\scriptsize$+$};
\node[inner sep=1] (b) at ({\s*3.55},0){\small $b$};
\node  at ({\s*4.33},0){$\big)$};
\path[thick,->,draw,Red]
(1) edge[bend left=60] (2)
(2) edge[bend left=60] (1)
;
\draw[->,thick,Red] (a.+112) .. controls ({\s*2.35-0.6},0.6) and ({\s*2.35+0.6},0.6) .. (a.+68);
\draw[->,thick,Red] (b.+112) .. controls ({\s*3.55-0.6},0.6) and ({\s*3.55+0.6},0.6) .. (b.+68);
\end{tikzpicture}
\end{array}
\]
{\caption{\label{fig:2C2} The two irreducible factorizations of $C_2+C_2$.}} 
\end{figure}

\medskip
As last general observation concerning $\S$, the division is not unique: given two sums of cycles $A,B$, the equation $AX=B$ can have several solutions $X$. For instance, Figure \ref{fig:2C2} shows that there exists two $X$ such that $C_2\cdot X=2C_2$, which are $X=C_2$ and $X=C_1+C_1$. Actually, we will prove that, for every $\epsilon>0$ and infinitely many $n$, there exists a sum of cycles $A$ with $n$ vertices and a sum of cycles $B$ with $n^2$ vertices such that the number of solutions $X$ to $AX=B$ is at least the expression given in \eqref{eq:super_poly}, super-polynomial in $n$.  

\medskip
From an algorithmic point of view, the \EM{division problem}, which consists, given $A,B$, to decide if $A$ divides $B$, is at the basis of many other problems \cite{riva2022factorisation,DFMR23,naquin2024factorisation,dennunzio2024note}. The main purpose of this paper is to study the complexity of this problem for sums of cycles. This restriction has already been considered in \cite{DFMR23} as an important step for solving polynomial equations over functional digraphs. On the other side, \cite{naquin2024factorisation} gives a polynomial algorithm to decide if $A$ divides $B$ when $B$ is a \EM{dendron}, that is, $B$ contains a unique cycle, of length $1$, so that $B$ consists of an out-tree plus a loop on the root. This result should be very useful to treat the transient part in the division problem. One may hope that efficient algorithms for the cyclic and transient parts could be combined to obtain an efficient algorithm for the general case. 

\medskip
Given two sum of cycles $A,B$, there is a simple sub-exponential algorithm to compute all the solutions $X$ of $AX=B$, that we call \EM{brute force approach}. It works as follows. Let $|A|$ and $|B|$ be the number of vertices in $A$ and $B$, respectively; $|A|+|B|$ is the size of the instance. If $AX=B$ then $X$ has $n=|B|/|A|$ vertices. Now remark that sums of cycles with $n$ vertices are in bijection with partitions of $n$: a sum of cycles $C_{\ell_1}+\cdots +C_{\ell_k}$ with $n$ vertices is completely described by the sequence $\ell_1,\dots,\ell_k$ of the length of its cycles, which form a partition of $n$; and conversely, the parts of a partition of $n$ describe the lengths of the cycles of a sum of cycles with $n$ vertices. For instance, $C_1+C_3$ corresponds to the partition $1+3$ of $4$. So to compute the solutions, we can enumerate the partitions of $n$, and check for each if the corresponding sum of cycles $X$ satisfies $AX=B$. This gives a sub-exponential algorithm. Indeed, the number of partitions of $n$ can be enumerated with linear-time delay and there are at most $e^{O(\sqrt{n})}$ such partitions, giving a total running time in $|B|e^{O(\sqrt{n})}$; details will be given in Section~\ref{sec:support}.

\medskip
Unfortunately, we were not able to find a faster algorithm, say running in~$|B|^d e^{O(n^\epsilon)}$ for some constants $d$ and $\epsilon<1/2$, to decide if $A$ divides $B$. That a polynomial algorithm exists is open, and \cite{dennunzio2024note} gives a positive answer under the strong condition that, in $A$ or $B$, all the cycles have the same length. 

\medskip
Here we design a general algorithm to compute the {\it number} of solutions which has the interesting property to be polynomial when $A$ is fixed. The precise statement, Theorem~\ref{thm:main} below, involves some definitions. The \EM{support} of a sum of cycles $A$ is the set \EM{$L(A)$} of positive integers $\ell$ such that $A$ contains $C_\ell$. Given $N\subseteq \N$, \EM{$\lcm N$} is the least common multiple of the integers in $N$, and \EM{$\div(n)$} is the number of divisors of~$n$.

\begin{theorem}\label{thm:main}
There is an algorithm that, given two non-empty sums of cycles $A,B$, computes the number of sums of cycles $X$ satisfying $AX=B$ with time complexity in 
\begin{equation}\label{eq:total_run}
O\left(|B|^3\left(\frac{|B|}{|A|}\right)^{\div(\lcm L(A))}\right).
\end{equation}
\end{theorem}

For the proof, we introduce two operations on an instance $(A,B)$ that we hope will be useful for further progress on the division problem. These are the \EM{split} and \EM{reduction} operations. The first splits $B$ into $B=B_1+B_2$ so that any solution of $(A,B)$ is obtained by adding a solution to $(A,B_1)$ with a solution to $(A,B_2)$. The second reduces $(A,B)$ into a smaller instance $(A',B')$ so that any solution of $(A,B)$ is obtained by multiplying the length of the cycles of a solution  of $(A',B')$ by some constant $d$. Repeating as much as  possible these operations, we obtain a \EM{decomposition} of $(A,B)$ into few smaller instances, which can be quickly solved with an easy improvement of the brute force approach, that we call \EM{brute force approach on the support}. The solutions of $(A,B)$ are then obtained with a simple combination of the solutions of the instances of its decomposition. In particular, the number of solutions of $(A,B)$ is simply the product of the number of solutions of the instances of the decomposition. 

%2 {sec:preliminaries}
%3 {sec:support}
%4 {sec:decomposition}
%4.1 {sec:split}
%4.2 {sec:reduction}
%5 {sec:principal_support}
%6 {sec:conclusion}
%A {an:irr_prime}
%B {an:many_irr}
%C {an:product}

\medskip
This paper is an extended version of the conference paper \cite{bridoux2024dividing}. It provides additional theoretical results, extended proofs of the main theorems, and a more detailed discussion of the implications and limitations of our approach. In particular, the content of the last two sections and the appendix, summarized briefly below, is new.

\paragraph{Organization of the paper} Section \ref{sec:preliminaries} gives the notations and basic results used throughout the paper. Section \ref{sec:support} presents in detail the brute force approach sketched above and its improvement, the brute force approach on the support, used in the algorithm of Theorem~\ref{thm:main}. Section \ref{sec:decomposition} presents the decomposition technique, involving the split and reduction operations. This decomposition is actually the main contribution, Theorem~\ref{thm:main} being a simple application. Section \ref{sec:principal_support} discusses an improvement of the brute force approach on the support, which should be a potential useful tool for further developments. To illustrate this, we easily prove that it gives, in polynomial time, all the solutions of instances $(A,B)$ such that $L(A)\setminus \{1\}$ only contains primes, or such that $1\in L(B)$. Concluding remarks are given in Section \ref{sec:conclusion}. The appendix gives the proofs of all the general observations made on $\S$ in this introduction; it is independent of the rest of the paper in that it only refers on the material given in Section \ref{sec:preliminaries}. In Appendix \ref{an:irr_prime}, we prove that almost all sums of cycles are irreducible and that $\S$ has no prime element. In Appendix \ref{an:many_irr} we prove that, for every $\epsilon>0$ and infinitely many $n$, there exists an instance $(A,B)$ with $|A|=n$ and $|B|=n^2$ such that the number of solutions of $(A,B)$ and the number of irreducible factorizations of $B$ is at least \eqref{eq:super_poly}, hence super-polynomial. Finally, in Appendix \ref{an:product}, we prove that a sum of cycles with $n$ vertices has at most $e^{O(\sqrt{n})}$ irreducible factorizations.

%%%%%%%%%%%%%%%%%%%%%%%%%%%%%%%%%%%%%%%%%%%%%%%%%%%%%%%%%%%%%%%%%%%%%%%%%%%%%%%%%%%%%%%%%
%%%%%%%%%%%%%%%%%%%%%%%%%%%%%%%%%%%%%%%%%%%%%%%%%%%%%%%%%%%%%%%%%%%%%%%%%%%%%%%%%%%%%%%%%
\section{Preliminaries}\label{sec:preliminaries}
%%%%%%%%%%%%%%%%%%%%%%%%%%%%%%%%%%%%%%%%%%%%%%%%%%%%%%%%%%%%%%%%%%%%%%%%%%%%%%%%%%%%%%%%%
%%%%%%%%%%%%%%%%%%%%%%%%%%%%%%%%%%%%%%%%%%%%%%%%%%%%%%%%%%%%%%%%%%%%%%%%%%%%%%%%%%%%%%%%%

\paragraph{Integers} 

Given $N\subseteq \N$, we denote by \EM{$\lcm N$} and \EM{$\gcd N$} the least common multiple and the greatest common divisor of the integers in $N$, respectively. For $n,m\in \N$, we set $\EM{n\lor m}=\lcm(n,m)$ and $\EM{n\land m}=\gcd(n,m)$, and for $N,M\subseteq \N$ we set $\EM{N\lor M}=\{n\lor m\mid n\in N,m\in M\}$. We denote by \EM{$\Div(n)$} and \EM{$\Mult(n)$} the set of divisors and positive multiples of $n$, respectively, and set $\EM{\div(n)}=|\Div(n)|$. We set $\EM{\Div(N)}=\bigcup_{n\in N}\Div(n)$, $\EM{\Mult(N)}=\bigcup_{n\in N}\Mult(n)$ and $\EM{\div(N)}=|\Div(N)|$.  For a positive integer $p$, we write \EM{$p\mid N$} to means that $p\mid n$ for all $n\in N$. We set $\EM{pN}=\{pn\mid n\in N\}$ and  $\EM{N/p}=\{n/p\mid n\in N, p\mid n\}$. For a positive integer $n$, and a prime $p$, we use the usual notation \EM{$\nu_p(n)$} to denote the greatest integer $\alpha$ such that $p^\alpha\mid n$.

\paragraph{Sums of cycles} 

The cycle of length $\ell$ is denoted \EM{$C_\ell$}, and the sum of cycles that consists of $n$ cycles of length $\ell$ is denoted by \EM{$nC_\ell$}. Let $A,B,X$ be sums of cycles. The number of vertices in $A$ is denoted \EM{$|A|$} and is the \EM{size} of $A$. The number of cycles of length $\ell$ in $A$ is denoted by \EM{$A(\ell)$}. Thus $A=\sum_{\ell\geq 1} A(\ell)C_\ell$ and 
\begin{equation}\label{eq:A_partition}
|A|=\sum_{\ell\geq 1}\ell A(\ell).
\end{equation}
The \EM{support} of $A$ is the set of $\ell$ such that $A(\ell)>0$, denoted \EM{$L(A)$}. We write \EM{$A\subseteq B$} to mean that $A$ is a subgraph of $B$, equivalently, $A(\ell)\leq B(\ell)$ for all $\ell\geq 1$. If $A\subseteq B$ then \EM{$B-A$} is the sum of cycles $A'$ such that $A+A'=B$. We easily check (see e.g \cite{dennunzio2018polynomial}) that the product between $C_a$ and $C_b$ is 
\[
C_aC_b=\left(\frac{ab}{a\lor b}\right)C_{a\lor b}=(a\land b)C_{a\lor b}.
\]
By distributivity, the product $AX$ satisfies, for all integers $\ell$,  
\begin{equation}\label{eq:mult}
(AX)(\ell)=\frac{1}{\ell}\sum_{\substack{a,x\in\N\\a\lor x=\ell}} aA(a)xX(x).
\end{equation}
Note that since $A(a)X(x)=0$ if $a\not\in L(A)$ or $x\not\in L(X)$, we can restrict the sum to the couples $(a,b)\in L(A)\times L(X)$ satisfying $a\lor x=\ell$ to get a finite expression. It is important to note that, by \eqref{eq:mult}, we have 
\begin{equation}\label{eq:LAB}
L(AX)=L(A)\lor L(X) 
\end{equation}
and thus 
\begin{equation}\label{eq:divLAX}
L(X)\subseteq \Div(L(AX)).
\end{equation}
If $A,B$ are non-empty, then $(A,B)$ is an \EM{instance}, and its \EM{size} is $|A|+|B|$. A \EM{solution} of $(A,B)$ is a sum of cycles $X$ such that $AX=B$. Note that $|X|=|B|/|A|$ for every solution $X$. We denote by \EM{$\Sol(A,B)$} the set of solutions, and $\EM{\sol(A,B)}=|\Sol(A,B)|$. If a solution exists, we say that $A$ \EM{divides} $B$ and we write \EM{$A\mid B$}.

\paragraph{Integer partitions} 

Let $n$ be a positive integer. A \EM{partition} of $n$ is a non-decreasing sequence of positive integers which sum to $n$. We denote by \EM{$p(n)$} the numbers of partition of $n$. As mentioned in the introduction, $p(n)$ is te number of sum of cycles with $n$ vertices; by \eqref{eq:A_partition}, we can regard a sum of cycles $A$ with $n$ vertices as a partition of $n$ where the number of occurences of each integer $\ell$ is $A(\ell)$. Let $\EMM{c_0}=\pi\sqrt{2/3}$. Hardy and Ramanujan \cite{hardy1918asymptotic} famously proved that $p(n)\sim (e^{c_0\sqrt{n}})/(4\sqrt{3}n)$. Here, we will only use the following bounds, proved by Erd\H{o}s (using only elementary methods) \cite{erdos1942elementary}: for every $\epsilon>0$, and $n$ large enough 
\begin{equation}\label{eq:Erdos_bounds}
(c_0-\epsilon)\sqrt{n} \leq \ln p(n)\leq c_0\sqrt{n}.
\end{equation}
A partition of $n$ with parts in $L$ is non-decreasing sequence of positive integers, all in $L$, which sum to $n$. We denote by \EM{$p_L(n)$} the number of partitions of $n$ with parts in $L$. Nathanson \cite{nathanson2000elementary} gives an asymptotic formula for $p_L(n)$ when $L$ is fixed, but we won't need it.

%%%%%%%%%%%%%%%%%%%%%%%%%%%%%%%%%%%%%%%%%%%%%%%%%%%%%%%%%%%%%%%%%%%%%%%%%%%%%%%%%%%%%%%%%
%%%%%%%%%%%%%%%%%%%%%%%%%%%%%%%%%%%%%%%%%%%%%%%%%%%%%%%%%%%%%%%%%%%%%%%%%%%%%%%%%%%%%%%%%
\section{Brute force approach on the support}\label{sec:support}
%%%%%%%%%%%%%%%%%%%%%%%%%%%%%%%%%%%%%%%%%%%%%%%%%%%%%%%%%%%%%%%%%%%%%%%%%%%%%%%%%%%%%%%%%
%%%%%%%%%%%%%%%%%%%%%%%%%%%%%%%%%%%%%%%%%%%%%%%%%%%%%%%%%%%%%%%%%%%%%%%%%%%%%%%%%%%%%%%%%

We saw in the introduction that there is a natural bijection between sums of cycles with $n$ vertices and partitions of $n$. This gives the following brute force approach to find all the solutions of an instance $(A,B)$:
\begin{quote}
{\bf Brute force approach.} If $n=|B|/|A|$ is not an integer, return $\emptyset$. Otherwise, enumerate all the partitions of $n$, select those which correspond to solutions, and return them.
\end{quote}
Let us analyse the complexity of this simple algorithm. By \eqref{eq:Erdos_bounds}, there are $e^{O(\sqrt{n})}$ partitions of $n$ to enumerate, and there are simple linear-time delay algorithms for this enumeration \cite{knuth2005art} \footnote{It is much more difficult to enumerate all the functional digraphs with $n$ vertices, and the recent paper \cite{defrain2024polynomial} provides a quadratic-time delay algorithm for this task.}. Hence the enumeration can be done in $e^{O(\sqrt{n})}$. To check if a partition corresponds to a solution, we take the sum of cycles $X$ corresponding to the partition, we compute the product $AX$ in $O(|B|)$, and we check if $AX=B$ in $O(|B|)$ \cite{hopcroft1974linear}. The total running time is thus 
\begin{equation}\label{eq:algo1}
|B|e^{O(\sqrt{n})}.
\end{equation}

%\begin{remark}
%We can easily check if $B$ is irreducible in $|B|e^{O(\sqrt{|B|})}$. Let $m=|B|$. 
% with the same time complexity. If $B$ reducible, it has a  For that, we enumerate for $2\leq m\leq \sqrt{n}$ all the partitions of $m$; this can be done in $e^{O(\sqrt{m})$. We tale the corresponding sum of cycles $A$, and check in $|B|e^{O(\sqrt{n})}$ if $A\mid B$. 
%\end{remark}

\medskip
Even if the only information that the algorithm extracts from the instance is the size $|B|/|A|$ of the solutions, it is actually not so easy to improve it significantly. Our approach to extract more information is to consider what we call the \EM{support} of the instance, which depends on the supports of $A$ and $B$ and the size $|B|/|A|$ of solutions. 

\begin{definition}
The \EM{support} of an instance $(A,B)$, denoted \EM{$L(A,B)$}, is the set of positive integers $x\leq |B|/|A|$ such that $L(A)\lor x\subseteq L(B)$. 
\end{definition}

\begin{example}\label{ex:1}
\[
\begin{array}{rcl}
L(C_6,C_6+C_{12})&=&\{1,2,3\}\\
L(C_6,3C_6+8C_{12})&=&\{1,2,3,4,6,12\}\\
L(C_6,6C_5+C_6)&=&\{1,2,3,6\}\\
\end{array}
\]
Note that $6,12\not\in L(C_6,C_6+C_{12})$ since $6,12 \geq | C_6+C_{12}|/|C_6| = 18/6 = 3$.
\end{example}

%We obviously have $L(A,B)\subseteq \Div(L(B))\subseteq \{1,\dots,b\}$ where $b$ is the maximum of $L(B)$. So we can compute $L(A,B)$ with time complexity $O(|A||B|^2)$ by selecting the integers $1\leq x\leq b$ such that $L(A)\lor x\subseteq L(B)$. Note also that  
%\begin{equation}\label{eq:LAB2}
%L(A)\lor L(A,B)\subseteq L(B).
%\end{equation}

The main property of $L(A,B)$ is that it bounds the support of any solution:
\begin{equation}\label{eq:support}
\forall X\in\Sol(A,B),\qquad L(X)\subseteq L(A,B).
\end{equation}
Indeed, if $AX=B$ then, by \eqref{eq:LAB}, we have $L(A)\lor L(X)=L(B)$. Since $\max L(X)\leq |X|=|B|/|A|$ we have $L(X)\subseteq L(A,B)$. In other words solutions are partitions of $n=|B|/|A|$ with parts in $L(A,B)$. This gives the following improvement to the previous algorithm:
\begin{quote}
{\bf Brute force approach on the support.} If $n=|B|/|A|$ is not an integer, return $\emptyset$. Otherwise, compute $L(A,B)$, enumerate all the partitions of $n$ with parts in $L(A,B)$, select those which correspond to solutions, and return them.
\end{quote}
Since $L(A,B)$ only contains integers between $1$ and $n=|B|/|A|$, we can compute $L(A,B)$ with time complexity $O(|B|^2)$ by enumerating the integers $1\leq x\leq n$ and selecting, in $O(|A||B|)$, those that satisfy $L(A)\lor x\subseteq L(B)$. Next, to enumerate the partitions of $n$ with parts in $L=L(A,B)$, we proceed as follows. In such a partition, the number of occurrences of $x\in L$ is between $0$ and $\lfloor n/x\rfloor$. Hence we can enumerate all the functions $f:L\to\N$ such that $f(x)\leq 1+\lfloor n/x\rfloor$ for all $x\in L$, and select those which correspond to partitions, that is, such that $\sum_{x\in L}xf(x)=n$. It is easy to show, by induction on $|L|$, that the number of such functions is at most $n^{|L|-1}$. Since the delay for the enumeration of these functions is linear, this gives an enumeration of the partitions of $n$ with parts in $L$ in $O(n^{|L|})$. Then, as before, we check if a partition corresponds to a solution in $O(|B|)$. The total complexity is thus $O(|B|^2+|B| n^{|L|})$, that we simplify in $O(|B|^2 n^{|L|})$. Hence we have proved the following.

\begin{lemma}\label{lem:brute}
The brute force approach on the support computes the set of solutions of $(A,B)$ in 
\begin{equation}\label{eq:algo2}
O\left(|B|^2\left(\frac{|B|}{|A|}\right)^{|L(A,B)|}\right).
\end{equation}
\end{lemma}

\section{Decomposition lemma}\label{sec:decomposition}
%%%%%%%%%%%%%%%%%%%%%%%%%%%%%%%%%%%%%%%%%%%%%%%%%%%%%%%%%%%%%%%%%%%%%%%%%%%%%%%%%%%%%%%%%
%%%%%%%%%%%%%%%%%%%%%%%%%%%%%%%%%%%%%%%%%%%%%%%%%%%%%%%%%%%%%%%%%%%%%%%%%%%%%%%%%%%%%%%%%

In this section, we state and prove our main result, which is the polynomial-time decomposition of an instance $(A,B)$ into smaller instances $(A_1,B_1),\dots,(A_k,B_k)$, so that the number of solutions $(A,B)$ can be reconstructed from the number of solutions of the smaller instances $(A_i,B_i)$. These smaller instances always exhibit a particular property: the size of their support is bounded according to $A_i$. Using the brute force approach on the support to compute the number of solutions of each $(A_i,B_i)$, we obtain the number of solutions for $(A,B)$ with a complexity similar to \eqref{eq:algo2}, but where the critical term $|L(A,B)|$ in exponent is, as for the smaller instances, bounded according to $A$. Hence the whole algorithm becomes polynomial when $A$ is fixed. 

\medskip
For the details we need some definitions. Let us say that an instance $(A,B)$ is \EM{consistent} if 
\[
L(A)\lor L(A,B)=L(B).
\]
Then non-consistent instances have no solution: if $AX=B$ then $(A,B)$ is consistent since 
\[
L(B)=L(AX)\overset{\eqref{eq:LAB}}{=}L(A)\lor L(X)\overset{\eqref{eq:support}}{\subseteq} L(A)\lor L(A,B)\subseteq L(B).
\]

\begin{example}\label{ex:2}
We take again the three instances of Example \ref{ex:1}.
\[
\begin{array}{rcllll}
L(C_6,C_6+C_{12})&=&\{1,2,3\}&&\textrm{non-consistent}:&\{6\}\lor \{1,2,3\}=\{6\}\neq \{6,12\}\\
L(C_6,3C_6+8C_{12})&=&\{1,2,3,4,6,12\}&&\textrm{consistent}: &\{6\}\lor \{1,2,3,4,6,12\}=\{6,12\}\\
L(C_6,6C_5+C_6)&=&\{1,2,3,6\}&&\textrm{non-consistent}:& \{6\}\lor \{1,2,3,6\}=\{6\}\neq \{5,6\}
\end{array}
\]
\end{example}

Let us say that an instance $(A,B)$ is \EM{basic} if 
\[
L(B)\subseteq \Div(\lcm L(A)).
\]
This is equivalent to say that for any prime $p$ and $b\in L(B)$, there exists $a\in L(A)$ such that $\nu_p(b)\leq\nu_p(a)$. Note that the support of a basic instance $(A,B)$ is bounded according to $A$ only:
\[
L(A,B)\subseteq \Div(L(B))\subseteq \Div(\lcm L(A)).
\]

\begin{example} The two instances below are consistent, but only one is basic: 
\[
\begin{array}{lll}
(C_2+C_6,C_{12})&\textrm{non-basic}:&L(C_{12})=\{12\}\not\subseteq \Div(\lcm \{2,6\})=\Div(6).\\
(C_5+C_6,C_6+C_{15})&\textrm{basic}:&L(C_6+C_{15})=\{6,15\}\subseteq \Div(\lcm \{5,6\})=\Div(30).
\end{array}
\]
\end{example}

\medskip
The \EM{cycle length multiplication} of $A$ by $p$, denoted \EM{$A\m p$}, is the sum of cycles obtained from $A$ by multiplying by $p$ the length of every cycle in $A$. In other words: for all $a\geq 1$, we have 
\[
(A\m p)(a)=
\left\{
\begin{array}{ll}
A(a/p)&\textrm{if }p\mid a\\
0&\textrm{otherwise.}
\end{array}
\right.
\]

\begin{example}
\[
(2C_1+3C_2+5C_3)\m 3=2C_3+3C_6+5C_9.
\]
\end{example}

\medskip
Given two sets of sums of cycles $\boldsymbol{A}$ and $\boldsymbol{B}$ we set 
\[
\EM{\boldsymbol{A}+\boldsymbol{B}}=\{A+B\mid A\in\boldsymbol{A}, B\in \boldsymbol{B}\},\qquad 
\EM{\boldsymbol{A}\m p}=\{A\m p\mid A\in \boldsymbol{A}\}.
\]

We are now ready to state the decomposition sketched at the beginning of the section. 

\begin{lemma}[{\bf Decomposition lemma}]\label{lem:main}
There is an algorithm that, given a consistent instance $(A,B)$, computes in $O(|B|^3)$ a list of $k\leq |B|$ basic instances $(A_1,B_2),\dots,(A_k,B_k)$ and positive integers $p_1,\dots,p_k$ such that: for all $1\leq i\leq k$, 
\begin{itemize}
\item 
$|A_i|=|A|$, 
\item
$\lcm L(A_i) \mid \lcm L(A)$, 
\item
$|B_1|+\cdots+|B_k|\leq |B|$, 
\item $\Sol(A,B)=(\Sol(A_1,B_1)\m p_1)+\cdots +(\Sol(A_k,B_k)\m p_k)$.
\end{itemize}
\end{lemma}

Before proving this lemma, let us show that it easily implies Theorem \ref{thm:main} given in the introduction, restated below. It shows that there is a general algorithm for computing the solutions of an instance $(A,B)$ with property to be polynomial when $A$ is fixed. 

\setcounter{theorem}{0}

\begin{theorem}
There is an algorithm that, given two non-empty sums of cycles $A,B$, computes the number of sums of cycles $X$ satisfying $AX=B$ with time complexity in 
\begin{equation*}
O\left(|B|^3\left(\frac{|B|}{|A|}\right)^{\div(\lcm L(A))}\right).
\end{equation*}
\end{theorem}

\begin{proof}[{\bf Proof of Theorem~\ref{thm:main} assuming Lemma~\ref{lem:main}.}]
The algorithm is as follows. First we check if $(A,B)$ is consistent. If $(A,B)$ is non-consistent, then $(A,B)$ has no solution and we output $0$. Otherwise, we compute the $k\leq |B|$ basic instances $(A_1,B_1),\dots,(A_k,B_k)$ with the algorithm of Lemma~\ref{lem:main}. Then, for all $1\leq i\leq k$, we use the brute force approach on the support to compute number $s_i$ of solutions of $(A_i,B_i)$, and we output the product $s_1\cdots s_k$, which is correct by the fourth bullet in Lemma~\ref{lem:main}. 

\medskip
Let us analyse the complexity. First, since $L(A,B)$ can be computed in $O(|B|^2)$ (see Section \ref{sec:support}), we can check if $(A,B)$ is consistent with the same complexity. Then, by Lemma \ref{lem:main}, the basic instances $(A_1,B_1),\dots,(A_k,B_k)$ are computed in $O(|B|^3)$. Let $1\leq i\leq k$. Using Lemma~\ref{lem:brute} and the fact that $(A_i,B_i)$ is basic, $s_i$ is computed in 
\begin{equation*}
O\left(|B_i|^2\left(\frac{|B_i|}{|A_i|}\right)^{\div(\lcm L(A_i))}\right).
\end{equation*}
By Lemma~\ref{lem:main}, we have $|A_i|=|A|$, $|B_i|\leq |B|$ and $\lcm L(A_i)$ divides $\lcm L(A)$. We can thus simplify the expression, with a loss of precision, as follows: $s_i$ is computed in 
\begin{equation*}
O\left(|B|^2\left(\frac{|B|}{|A|}\right)^{\div(\lcm L(A))}\right).
\end{equation*}
Since $k\leq |B|$ we obtain the global time-complexity of the statement. 
\end{proof}

The rest of the section is devoted to the proof of Lemma~\ref{lem:main}.

%%%%%%%%%%%%%%%%%%%%%%%%%%%%%%%%%%%%%%%%%%%%%%%%%%%%%%%%%%%%%%%%%%%%%%%%%%%%%%%%%%%%%%%%%
%%%%%%%%%%%%%%%%%%%%%%%%%%%%%%%%%%%%%%%%%%%%%%%%%%%%%%%%%%%%%%%%%%%%%%%%%%%%%%%%%%%%%%%%%
\subsection{Split}\label{sec:split}
%%%%%%%%%%%%%%%%%%%%%%%%%%%%%%%%%%%%%%%%%%%%%%%%%%%%%%%%%%%%%%%%%%%%%%%%%%%%%%%%%%%%%%%%%
%%%%%%%%%%%%%%%%%%%%%%%%%%%%%%%%%%%%%%%%%%%%%%%%%%%%%%%%%%%%%%%%%%%%%%%%%%%%%%%%%%%%%%%%%

The decomposition lemma involves two operations. The first is the split, the second is the reduction. In this subsection, we treat the split, the reduction is the subject of the next subsection. 

\medskip
A \EM{split} of an instance $(A,B)$ is a couple of instances $(A,B_1),(A,B_2)$ such that $B=B_1+B_2$; it is \EM{consistent} if the two instances are. A key observation is that if $(A,B_1),(A,B_2)$ is a split of $(A,B)$ then
\begin{equation}
\Sol(A,B_1)+\Sol(A,B_2)\subseteq \Sol(A,B)
\end{equation}
Indeed, if $AX_1=B_1$ and $AX_2=B_2$ then $A(X_1+X_2)=AX_1+AX_2=B_1+B_2=B$. If the inclusion above is an equality, we say that the split is \EM{perfect}. We can then reconstruct the solutions of $(A,B)$ from that of $(A,B_1)$ and $(A,B_2)$, which is obviously interesting from an algorithmic point of view. 

\medskip
Here is a simple sufficient condition for an instance to admit a consistent perfect split. 

\begin{lemma}\label{lem:partition}
Let $(A,B_1),(A,B_2)$ be a split of a consistent instance $(A,B)$. Suppose that $L(B_1)$ and $L(B_2)$ are disjoint, and suppose that $L(A,B_1),L(A,B_2)$ is a partition of $L(A,B)$. Then $(A,B_1),(A,B_2)$ is a consistent perfect split of $(A,B)$. 
\end{lemma}

\begin{proof}
For $i=1,2$, let $L_i=L(A,B_i)$. By definition, $L(A)\lor L_i\subseteq L(B_i)$. Since $(A,B)$ is consistent, and $L(A,B)=L_1\cup L_2$ we have  
\[
L(B)=L(A)\lor L(A,B)=(L(A)\lor L_1)\cup (L(A)\lor L_2)\subseteq L(B_1)\cup L(B_2)=L(B).
\]
Thus the above inclusion is actually an equality. Since $L(A)\lor L_i\subseteq L(B_i)$ and $L(B_1)\cap L(B_2)=\emptyset$ we deduce that $L(A)\lor L_i=L(B_i)$. Hence the split $(A,B_1),(A,B_2)$ is consistent. It remains to prove that it is perfect. 

\medskip
As already mentioned, if $X_1,X_2$ are solutions of $(A,B_1),(A,B_2)$ then
\[
A(X_1+X_2)=AX_1+AX_2=B_1+B_2=B,
\]
thus $X=X_1+X_2$ is a solution of $(A,B)$. 

\medskip
Conversely, let $X$ be a solution of $(A,B)$ and let us prove that $X=X_1+X_2$ for some solutions $X_1,X_2$ of $(A,B_1),(A,B_2)$. Let $X_i$ be the sum of cycles which contains exactly all the cycles of $X$ whose length is in $L_i$. Since $L_1\cap L_2=\emptyset$, we have $L(X_1)\cap L(X_2)=\emptyset$. By \eqref{eq:support} we have $L(X)\subseteq L(A,B)$ and since $L_1,L_2$ is a partition of $L(A,B)$ we have $X=X_1+X_2$. Using \eqref{eq:LAB} we obtain 
\begin{eqnarray*}
L(B)=L(AX)&=&L(A)\lor L(X)\\
&=&(L(A)\lor L(X_1))\cup (L(A)\lor L(X_2))=L(AX_1)\cup L(AX_2).
\end{eqnarray*}
For $i=1,2$, we have $L(X_i)\subseteq L_i$ and thus, using \eqref{eq:LAB},   
\[
L(AX_i)=L(A)\lor L(X_i)\subseteq L(A)\lor L_i= L(B_i).
\]
Since $L(B_1),L(B_2)$ is a partition of $L(B)$, we deduce that $L(AX_i)=L(B_i)$ for $i=1,2$. Hence, to prove that $X_i$ is a solution of $(A,B_i)$, it is sufficient to prove that $AX_i(b)=B_i(b)$ for all $b\in L(B_i)$. So let $b\in L(B_i)$. Since $L(B_1)\cap L(B_2)=\emptyset$, for every $a\in L(A)$ and $x\in L(X)$ with $a\lor x=b$ we have $x\in L_i$ and thus $x\in L(X_i)$. Furthermore, since $L(X_1)\cap L(X_2)=\emptyset$, we have $X_i(x)=X(x)$. Consequently, 
\begin{eqnarray*}
(AX_i)(b)
	=\frac{1}{b}\sum_{\substack{a\in L(A)\\x\in L(X_i)\\a\lor x=b}}aA(a)xX_i(x)
	&=&\frac{1}{b}\sum_{\substack{a\in L(A)\\x\in L(X_i)\\a\lor x=b}}aA(a)xX(x)\\
	&=&\frac{1}{b}\sum_{\substack{a\in L(A)\\x\in L(X)\\a\lor x=b}}aA(a)xX(x)=B(b)=B_i(b).\\
\end{eqnarray*}
\end{proof}

We now prove that a non-basic consistent instance $(A,B)$ with $\gcd L(A,B)=1$ admits a consistent perfect split; we will then prove, with the reduction operation (next subsection), that, in some sense, the condition on the $\gcd$ can be suppressed, leading to a consistent prefect split of every non-basic consistent instance.  

\begin{lemma}\label{lem:partition2}
Let $(A,B)$ be a non-basic consistent instance with $\gcd L(A,B)=1$. Let $b\in L(B)$ and a prime $p$ such that $\nu_p(b)>\nu_p(a)$ for all $a\in L(A)$; $b$ and $p$ exist since $(A,B)$ is not basic. Let $B_1$ be the sum of cycles which contains exactly all the cycles of $B$ whose length $\ell$ satisfies $\nu_p(\ell)=\nu_p(b)$, and let $B_2=B-B_1$. Then $(A,B_1),(A,B_2)$ is a consistent perfect split of $(A,B)$. 
\end{lemma}

\begin{proof}
By construction we have $L(B_1)\cap L(B_2)=\emptyset$. Let $x\in L(A,B)$. If $\nu_p(x)=\nu_p(b)$ then, for all $\ell\in (L(A)\lor x)$ we have $\nu_p(\ell)=\nu_p(b)$. Since $\ell\in L(B)$ we deduce that $\ell\in L(B_1)$. Consequently, $x\in L(A,B_1)$.  If $\nu_p(x)\neq \nu_p(b)$ then, for all $\ell\in (L(A)\lor x)$ we have $\nu_p(\ell)\neq \nu_p(b)$. Since $\ell\in L(B)$ we deduce that $\ell\in L(B_2)$. Consequently, $x\in L(A,B_2)$. This proves that $L(A,B_1),L(A,B_2)$ is a partition of $L(A,B)$. This also proves that $p\mid \gcd L(A,B_1)$, and since  $\gcd L(A,B)=1$ we deduce that $L(A,B_2)\neq\emptyset$, and thus $B_2$ is non-empty. So $(A,B_1),(A,B_2)$ is a split of $(A,B)$, which is consistent and perfect by Lemma \ref{lem:partition}.
\end{proof}

The example below illustrates this lemma. 

\begin{example}\label{ex:split}
Let $A=C_6$ and $B=3C_6+8C_{12}$. Then $(A,B)$ is consistent, but obviously not basic, and $\gcd L(A,B) =1$ (see Ex. \ref{ex:2}). Applying Lemma~\ref{lem:partition2} with $b=12$ and $p=2$ we obtain the consistent prefect split $(A,B_1),(A,B_2)$ with $B_1=8C_{12}$ and $B_2=3C_6$. Since $L(A,B_1)=\{4,12\}$ we have 
\begin{eqnarray*}
AX_1=B_1&\iff&  C_6\big(X_1(4)C_4+X_1(12)C_{12}\big)=8C_{12}\\
&\iff&  2X_1(4)C_{12}+6X_1(12)C_{12}=8C_{12}\\
&\iff &2X_1(4)+6X_1(12)=8.
\end{eqnarray*}
Thus each solution $X_1$ to $(A,B_1)$ corresponds to a partition of $8$ with parts in $\{2,6\}$: these are $2+2+2+2$ and $2+6$, giving the following two solutions:
\[
\begin{array}{l}
4C_4\\
C_4+C_{12}.
\end{array}
\]
Since $L(A,B_2)=\{1,2,3,6\}$ we have 
\begin{eqnarray*}
AX_2=B_2&\iff&  C_6\big(X_2(1)C_1+X_2(2)C_2+X_2(3)C_3+X_2(6)C_6\big)=3C_6\\
&\iff&  X_2(1)C_6+2X_2(2)C_6+3X_2(3)C_6+6X_2(6)C_6=3C_6\\
&\iff&  X_2(1)+2X_2(2)+3X_2(3)+6X_2(6)=3.
\end{eqnarray*}
Thus each solution $X_2$ to $(A,B_2)$ corresponds to a partition of $3$ with parts in $\{1,2,3,6\}$: these are $1+1+1$, $1+2$, and $3$, giving the following three solutions:
\[
\begin{array}{l}
3C_1\\
C_1+C_2\\
C_3.
\end{array}
\]
Since the split is perfect, we have $\Sol(A,B)=\Sol(A,B_1)+\Sol(A,B_2)$ and thus $(A,B)$ has the following $6$ solutions:
\[
\begin{array}{l}
3C_1+4C_4\\
3C_1+C_4+C_{12}\\
C_1+C_2+4C_4\\
C_1+C_2+C_4+C_{12}\\
C_3+4C_4\\
C_3+C_4+C_{12}.
\end{array}
\]  
\end{example}

%%%%%%%%%%%%%%%%%%%%%%%%%%%%%%%%%%%%%%%%%%%%%%%%%%%%%%%%%%%%%%%%%%%%%%%%%%%%%%%%%%%%%%%%%
%%%%%%%%%%%%%%%%%%%%%%%%%%%%%%%%%%%%%%%%%%%%%%%%%%%%%%%%%%%%%%%%%%%%%%%%%%%%%%%%%%%%%%%%%
\subsection{Reduction}\label{sec:reduction}
%%%%%%%%%%%%%%%%%%%%%%%%%%%%%%%%%%%%%%%%%%%%%%%%%%%%%%%%%%%%%%%%%%%%%%%%%%%%%%%%%%%%%%%%%
%%%%%%%%%%%%%%%%%%%%%%%%%%%%%%%%%%%%%%%%%%%%%%%%%%%%%%%%%%%%%%%%%%%%%%%%%%%%%%%%%%%%%%%%%

We say that an instance $(A,B)$ is \EM{compact} if $\gcd L(A,B)=1$. This condition is used in Lemma~\ref{lem:partition2} to split non-basic instances, but in this section we show that every instance can be reduced to an ``equivalent'' compact instance, which can then be splitted. For this reduction, we need two operations. 

\medskip
The first operation is the cycle length division. The \EM{cycle length division} of a sum of cycles $A$ by a positive integer $p$, denoted \EM{$A\d p$}, is the sum of cycles obtained from $A$ by deleting every cycle whose length is not a multiple of $p$, and by dividing by $p$ the length of the remaining cycles; in other words: for all $a\geq 1$, 
\[
(A\d p)(a)=A(pa).
\]
Note that $L(A\d p)=L(A)/p$ and if  $p\mid L(A)$ then $L(A)=p L(A\d p)$. 

\begin{example}
\[
(2C_1+3C_3+5C_4+7C_6)\d 3=3C_1+7C_2.
\]
\end{example}

The second operation is the contraction. The \EM{contraction} of a sum of cycles $A$ by a positive integer $p$, denoted \EM{$A\dd p$}, is inductively defined as follows. Firstly, $A\dd 1=A$. Secondly, if $p$ is prime then, for all $a\geq 1$, 
\begin{equation}\label{eq:dd}
(A\dd p)(a)=
\left\{
\begin{array}{ll}
A(a)+pA(pa)&\textrm{if }p\nmid a\\
pA(pa)&\textrm{otherwise.}
\end{array}
\right.
\end{equation}
This operation transforms each cycle of length $pa$ into $p$ cycles of length $a$ (and thus keeps the number of vertices unchanged). Note that $L(A\dd p)$ is the set of integers $a$ such that either $a\in L(A)$ and $p\nmid a$ or $pa\in L(A)$. Finally, if $p$ is composite, we take the smallest prime $q$ that divides $p$ and set 
\[
A\dd p=(A\dd p/q)\dd q.
\]

\begin{example}
\begin{eqnarray*}
(2C_1+3C_3+5C_4+ 7C_6)\dd 3&=&2C_1+9C_1+5C_4+21C_2\\&=&11C_1+5C_4+21C_2.
\end{eqnarray*}
\begin{eqnarray*}
(2C_1+3C_3+5C_4+ 7C_6)\dd 6&=&((2C_1+3C_3+5C_4+ 7C_6)\dd 3)\dd 2\\
&=&(11C_1+5C_4+21C_2)\dd 2\\
&=&(11C_1+10C_2+42C_1)\\
&=&(53C_1+10C_2).
\end{eqnarray*}
\end{example}

Note that, for every positive integers $p,q$, we have 
\begin{equation}\label{eq:pq}
(A\m p)\m q=A\m pq,\quad (A\d p)\d q=A\d pq,\quad (A\dd p)\dd q=A\dd pq.
\end{equation}
The first two equalities are obvious. The third results from the following easy to check commutativity property: $(A\dd p)\dd q=(A\dd q)\dd p$ when $p$ and $q$ are~primes. 

\medskip
We can finally define the reduction operation. The \EM{reduction} of a consistent instance $(A,B)$~is 
\[
(A\dd \gcd L(A,B),B\d \gcd L(A,B)).
\]
The main property of this reduction, stated below, is that it preserves consistency and results in compact instances without loss of information concerning the solutions. 

\begin{lemma}\label{lem:instance-full-reduction}
Let $(A,B)$ be a consistent instance. Its reduction $(A',B')$ is consistent and satisfies 
\[
L(A',B')=L(A,B)/\gcd L(A,B)\quad\textrm{and}\quad \Sol(A,B)=\Sol(A',B')\m \gcd L(A,B).
\]
\end{lemma}

Before proving this lemma, we illustrate it with an example. 

\begin{example}\label{ex:reduction}
The support of $(C_6,8C_{12})$ is $\{4,12\}$ and $\lcm\{4,12\}=4$. We have 
\[
C_6\dd 4=(C_6\dd 2)\dd 2=2C_3\dd 2=2C_3
\]
and 
\[
8C_{12}\d 4=8C_3.
\]
Thus the reduction of $(C_6,8C_{12})$ is $(2C_3,8C_3)$. Hence the support of the reduction is $\{1,3\}=\{4,12\}/4$, as predicted by Lemma~\ref{lem:instance-full-reduction}, so the reduction is compact, and we have 
\begin{eqnarray*}
2C_3X=8C_3&\iff&  2C_3\big(X(1)C_1+X(3)C_3\big)=8C_3\\
&\iff&  2X(1)C_3+6X(3)C_3=8C_3\\
&\iff &2X(1)+6X(3)=8.
\end{eqnarray*}
Thus each solution $X$ corresponds to a partition of $8$ with parts in $\{2,6\}$: these are $2+2+2+2+2$ and $2+6$, giving the following two solutions:
\[
\begin{array}{l}
4C_1\\
C_1+C_3.
\end{array}
\]
By Lemma \ref{lem:instance-full-reduction}, we have $\Sol(C_6,8C_{12})=\Sol(2C_3,8C_3)\m 4$. Hence the solutions of $(C_6,8C_{12})$ are 
\[
\begin{array}{lll}
(4C_1)\m 4&=&4C_4\\
(C_1+C_3)\m 4&=&C_4+C_{12}
\end{array}
\]
which is consistent with the direct computation given in Ex.~\ref{ex:split}. 
\end{example}

The rest of this subsection is devoted to the proof of Lemma \ref{lem:instance-full-reduction}. We first show that the cycle length division $\d$ is the inverse of the cycle length multiplication $\m$.

\begin{lemma}\label{lem:inversion}
Let $A$ be a sum of cycles and let $p$ be a positive integer. Then $(A\m p)\d p=A$, and if $p\mid L(A)$ then $(A\d p)\m p=A$.
\end{lemma}

\begin{proof}
For all $a\geq 1$, we have $((A\m p)\d p)(a)=(A\m p)(pa)=A(a)$. Suppose that $p\mid L(A)$ and let $a\geq 1$. If $p\nmid a$ then $A(a)=0$ and $((A\d p)\m p)(a)=0$ (since $p\mid L((A\d p)\m p)$). If $p\mid a$ then $((A\d p)\m p)(a)=(A\d p)(a/p)=A(a)$.  
\end{proof}

The key argument follows. 

\begin{lemma}\label{lem:product-reduction}
Let $A,X$ be sums of cycles, and suppose that $p\mid L(X)$ for some prime $p$. Then 
\[
(A\dd p)(X\d p)=(AX)\d p.
\]
\end{lemma}

\begin{proof}
Let $A'=A\dd p$ and $X'=X\d p$. We have to prove that $A'X'=AX\d p$, that is, for all $\ell\geq 1$, $A'X'(\ell)=(AX\d p)(\ell)=AX(p\ell)$. Let us fix $\ell\geq 1$. We have 
\[
p\ell A'X'(\ell)
	=\sum_{\substack{a,x\\a\lor x=\ell}} paA'(a)xX'(x)
	=\sum_{\substack{a,x\\a\lor x=\ell}} aA'(a)pxX(px). 
\]
Denoting by $\Omega$ the set of couples $(a,x)\in\N^2$ with $p\mid x$ and $a\lor \frac{x}{p}=\ell$, we obtain 
\[
p\ell A'X'(\ell)
	=\sum_{(a,x)\in\Omega} aA'(a)xX(x). 
\] 
By splinting the sum according to the definition of $A'$ we obtain
\begin{eqnarray*}
p\ell A'X'(\ell)
	&=&\sum_{\substack{(a,x)\in\Omega\\p\nmid a}} \big(aA(a)+paA(pa)\big)xX(x)+
	\sum_{\substack{(a,x)\in\Omega\\ p\mid a}} paA(pa)xX(x)\\
	&=&\sum_{\substack{(a,x)\in\Omega\\ p\nmid a}} aA(a)xX(x)
	+\sum_{\substack{(a,x)\in\Omega\\ p\nmid a}} paA(pa)xX(x)
	+\sum_{\substack{(a,x)\in\Omega\\ p\mid a}} paA(pa)xX(x). 
\end{eqnarray*}
Denoting by $\Omega'$ the set of $(a,x)\in\N^2$ with $p\mid x$, $p\mid a$ and $\frac{a}{p}\lor \frac{x}{p}=\ell$, we obtain 
\[
p\ell A'X'(\ell)
	=\sum_{\substack{(a,x)\in\Omega\\ p\nmid a}} aA(a)xX(x)
	+\sum_{\substack{(a,x)\in\Omega'\\p\nmid \frac{a}{p}}} aA(a)xX(x)
	+\sum_{\substack{(a,x)\in\Omega'\\p\mid \frac{a}{p}}} aA(a)xX(x). 
\] 
If $p\nmid a$ then $a\lor \frac{x}{p}=\ell$ iff $a\lor x=p\ell$; and $\frac{a}{p}\lor \frac{x}{p}=\ell$ iff $a\lor x=p\ell$. Consequently
\[
p\ell A'X'(\ell)
	=\sum_{\substack{a,x\\p\mid x\\ a\lor x=p\ell}} aA(a)xX(x).
\] 
Since $p\mid L(X)$, if $p\nmid x$ then $X(x)=0$, so
\[
p\ell A'X'(\ell)
	=\sum_{\substack{a,x\\ a\lor x=p\ell}} aA(a)xX(x)=p\ell AX(p\ell). 
\] 
Thus $A'X'(\ell)=AX(p\ell)$ for all $\ell\geq 0$, as desired. 
\end{proof}  

We obtain that every non compact instance can be partially reduced by a prime. 

\begin{lemma}\label{lem:instance-reduction}
Let $(A,B)$ be a consistent instance, and suppose $p\mid L(A,B)$ for some prime $p$. Then $(A\dd p,B\d p)$ is consistent with support $L(A,B)/p$, and 
\begin{equation}\label{eq:instance-reduction}
\Sol(A,B)=\Sol(A\dd p,B\d p)\m p.
\end{equation}
\end{lemma}

%\begin{lemma}\label{lem:instance-reduction}
%Let $(A,B)$ be a consistent instance, and suppose that there is a prime $p\mid L(A,B)$. Then $(A\dd p,B\d p)$ is consistent with support $L(A,B)/p$ and 
%\[
%\Sol(A,B)=\Sol(A\dd p,B\d p)\m p.
%\]
%\end{lemma}

\begin{proof}
Let $A'=A\dd p$ and $B'=B\d p$. We first prove \eqref{eq:instance-reduction}. Let $X$ be a solution of $(A,B)$. By  \eqref{eq:support} we have $L(X)\subseteq L(A,B)$ and since $p\mid L(A,B)$ we have $p\mid L(X)$. Hence, by Lemma~\ref{lem:product-reduction}, $A'(X\d p)=AX\d p=B\d p=B'$, that is, $X\d p$ is a solution of $(A',B')$. Since $p\mid L(X)$, by Lemma 
\ref{lem:inversion} we have $(X\d p)\m p=X$ and thus $X\in \Sol(A'\,B')\m p$. 

\medskip
We now prove the converse direction. Let $X'$ be a solution of $(A',B')$, and let $X=X'\m p$. We have to prove that $X$ is a solution of $(A,B)$. By Lemma  \ref{lem:inversion} we have $X\d p=X'$ thus $A'(X\d p)=B'=B\d p$. Since $p\mid L(X)$, by Lemma~\ref{lem:product-reduction}, we have $A'(X\d p)=(AX)\d p$. Thus $(AX)\d p=B\d p$. Since $p\mid L(A,B)$ and $(A,B)$ is consistent, we have $p\mid L(B)$. Since $X=X'\m p$ we obviously have $p\mid L(X)$. Thus $p$ divides $L(A)\lor L(X)=L(AX)$. Using Lemma~\ref{lem:inversion} we obtain $AX=((AX)\d p)\m p=(B\d p)\m p=B$. Thus $X$ is a solution of $(A,B)$. This proves~\eqref{eq:instance-reduction}.

\medskip
We now prove that $L(A,B)/p\subseteq L(A',B')$. For that, we fix $x\in L(A,B)/p$, and we prove that $a\lor x$ is in $L(B')$ for any $a\in L(A')$. Indeed, if $pa\in L(A)$ then $pa\lor px=b$ for some $b\in L(B)$ and we deduce that $a\lor x=b/p\in L(B')$. If $pa\not\in L(A)$, then $p\nmid a$ and $a\in L(A)$. Thus $a\lor px=b$ for some $b\in L(B)$ and since $p\nmid a$ we deduce that $a\lor x=b/p\in L(B')$. 

\medskip
We now prove the converse inclusion. For that, we fix $x\in L(A',B')$, and we prove that $a\lor px$ is in $L(B)$ for any $a\in L(A)$. Indeed, if $p\mid a$ then $a/p\in L(A')$ and thus $(a/p)\lor x=b$ for some $b\in L(B')$ so that $a\lor px=pb\in L(B)$. If $p\nmid a$ then $a\in L(A')$ and thus $a\lor x=b$ for some $b\in L(B')$, and since $p\nmid a$ we have $a\lor px=pb\in L(B)$.

\medskip
We finally prove that $(A',B')$ is consistent. Since $L(A')\lor L(A',B')\subseteq L(B')$, we only have to prove the converse inclusion. Let $b\in L(B')$. Then $pb\in L(B)$ and since $(A,B)$ is consistent, there is $a\in L(A)$ and $x\in L(A,B)$ with $a\lor x=pb$. Hence $x/p\in L(A',B')$. If $p\mid a$ then $(a/p)\lor (x/p)=b$ and we are done since $a/p\in L(A')$. If $p\nmid a$ then $a\lor (x/p)=b$ and we are done since $a\in L(A')$. 
\end{proof}

The proof of Lemma \ref{lem:instance-full-reduction} is then obtained by applying several times the previous lemma.

\begin{proof}[{\bf Proof of Lemma \ref{lem:instance-full-reduction}}]
Let $d=\gcd L(A,B)$. Suppose that $d>1$ since otherwise there is nothing to prove. Let us write $d$ as the product of $k\geq 1$ primes,  not necessarily distinct, say $d=p_1p_2\dots p_k$ with $p_1\leq p_2\leq \cdots\leq p_k$. Let $A_0=A$, $B_0=B$ and, for $1\leq\ell\leq k$, let  $A_\ell=A_{\ell-1} \dd p_\ell$ and $B_\ell=B_{\ell-1} \d p_\ell$. By Lemma \ref{lem:instance-reduction}, $(A_\ell,B_\ell)$ is a consistent instance and $\Sol(A_{\ell-1},B_{\ell-1})=\Sol(A_\ell,B_\ell)\m p_\ell$. By \eqref{eq:pq} we have $A_k=A\dd d$, $B_k=B\dd d$, $L(A_k,B_k)=L(A,B)/d$ and $\Sol(A,B)=\Sol(A_k,B_k)\m d$.
\end{proof}

%%%%%%%%%%%%%%%%%%%%%%%%%%%%%%%%%%%%%%%%%%%%%%%%%%%%%%%%%%%%%%%%%%%%%%%%%%%%%%%%%%%%%%%%%
%%%%%%%%%%%%%%%%%%%%%%%%%%%%%%%%%%%%%%%%%%%%%%%%%%%%%%%%%%%%%%%%%%%%%%%%%%%%%%%%%%%%%%%%%
\subsection{Proof of the decomposition lemma}
%%%%%%%%%%%%%%%%%%%%%%%%%%%%%%%%%%%%%%%%%%%%%%%%%%%%%%%%%%%%%%%%%%%%%%%%%%%%%%%%%%%%%%%%%
%%%%%%%%%%%%%%%%%%%%%%%%%%%%%%%%%%%%%%%%%%%%%%%%%%%%%%%%%%%%%%%%%%%%%%%%%%%%%%%%%%%%%%%%%

In this subsection, we finally prove Lemma~\ref{lem:main}. It is actually more convenient to prove something stronger. We start with a definition. Let $(A,B)$ be a consistent instance. A \EM{decomposition} of $(A,B)$ is a list $\L$ of triples $(A_i,B_i,p_i)$, $1\leq i\leq k$, such that 
\begin{itemize}
\item[$\bullet$] $(A_i,B_i)$ is a compact and consistent instance, and $p_i$ is a positive integer,
\item[$\bullet$] $|A_i|=|A|$ and $\lcm L(A_i)\mid \lcm L(A)$,
\item[$\bullet$] $|B_1|+\cdots +|B_k|\leq |B|$, 
\item[$\bullet$] $\Sol(A,B)=(\Sol(A_1,B_1)\m p_1)+\cdots +(\Sol(A_k,B_k)\m p_k)$.
\end{itemize}
We call $k$ the \EM{length} of $\L$; note that by the third point, $k\leq |B|$. Furthermore, we say that $\L$ is \EM{basic} if $(A_i,B_i)$ is basic for all $1\leq i\leq k$. We will prove that we can compute in $O(|A||B|^2)$ a basic decomposition, which clearly implies Lemma~\ref{lem:main}. For that we first prove that if $(A,B)$ has a non-basic decomposition, we can obtain a longer decomposition by partitioning a non-basic instance (Lemma~\ref{lem:partition2}) and then reducing its parts (Lemma \ref{lem:instance-full-reduction}). 

\begin{lemma}\label{lem:decomposition}
There is an algorithm that, given a consistent instance $(A,B)$ and a non-basic decomposition $\L$ of $(A,B)$ of length $k$, computes in $O(|B|^2)$ a decomposition $\L'$ of $(A,B)$ of length $k+1$.
\end{lemma}

\begin{proof}
The algorithm is as follows. Let $(A_1,B_1,p_1),\dots,(A_k,B_k,p_k)$ be the triples of $\L$. We can test if $(A_i,B_i,p_i)$ is basic or not in $O(|A_i|+|B_i|)=O(|B_i|)$. Since $k\leq |B|$ (by the third point of the definition of a decomposition), and since $\L$ is not basic, we can thus find a non-basic instance $(A_i,B_i)$ in $O(|B|^2)$. Since $(A_i,B_i)$ is compact and consistent, by Lemma~\ref{lem:partition2}, it has a consistent perfect split, say $(A_i,B_{i1}),(A_i,B_{i2})$. For $j=1,2$, we compute in $O(|B|^2)$ the integer $p_{ij}=\gcd L(A_i,B_{ij})$ and the reduction $(A_{ij},B'_{ij})=(A_i\dd p_{ij},B_{ij}\d p_{ij})$ of $(A_i,B_{ij})$. Finally, we output the list $\L'$ of length $k+1$ obtained from $\L$ by deleting $(A_i,B_i,p_i)$ and adding $(A_{i1},B'_{i1},p_ip_{i1})$ and $(A_{i2},B'_{i2},p_ip_{i2})$. So the running time is $O(|B|^2)$. 

\medskip
Let us prove that $\L'$ is a decomposition. By Lemma \ref{lem:instance-full-reduction}, $(A_{ij},B'_{ij})$ is compact and consistent. Furthermore, $|A_{ij}|=|A_i|=|A|$ and since $A_{ij}$ is a contraction of $A_i$, each member of $L(A_{ij})$ divides some member of $L(A_i)$, thus $\lcm L(A_{ij})$ divides $\lcm L(A_i)$, which divides $\lcm L(A)$. Thus $\lcm L(A_{ij})$ divides $\lcm L(A)$. Next, since $|B'_{i1}|+|B'_{i2}|\leq |B_{i1}|+|B_{i2}|=|B_i|$, the third point of the definition of a decomposition is preserved. Finally, by Lemma~\ref{lem:instance-full-reduction}, $\Sol(A_{i},B_{ij})=\Sol(A_{ij},B'_{ij})\m p_{ij}$. Since the split $(A_i,B_{i1}),(A_i,B_{i2})$ is perfect, we obtain 
\begin{eqnarray*}
\Sol(A_i,B_i)\m p_i&=&(\Sol(A_i,B_{i1})+\Sol(A_i,B_{i2}))\m p_i\\
&=&(\Sol(A_i,B_{i1})\m p_i) +(\Sol(A_i,B_{i2}) \m p_i)\\
&=&(\Sol(A_{i1},B'_{i1})\m p_ip_{i1}) +(\Sol(A_{i2},B'_{i2}) \m p_ip_{i2})
\end{eqnarray*}
and this proves that the last point of the definition of a decomposition is preserved. So $\L'$ is indeed a decomposition of $(A,B)$. 
\end{proof}

Iterating the previous lemma, we get the following, which implies Lemma~\ref{lem:main}. 
 
\begin{lemma}\label{lem:decompo2}
There is an algorithm that, given a consistent instance $(A,B)$, computes in $O(|B|^3)$ a basic decomposition of $(A,B)$. 
\end{lemma}

\begin{proof}
The algorithm constructs recursively a list $\L_1,\dots,\L_{|B|}$ of decompositions of $(A,B)$, where the length of $\L_r$ is at most $r$, and output $\L_{|B|}$. First we compute $\L_1=\{(A\dd d,B\d d,d)\}$ where $d=\gcd L(A,B)$ in $O(|A||B|)$; by Lemma~\ref{lem:instance-full-reduction}, $\L_1$ is a decomposition of length one. Now, suppose that the decomposition $\L_r$ of length $k\leq r< |B|$ has already been computed. If $\L_r$ is basic, we set $\L_{r+1}=\L_r$. Otherwise, using Lemma~\ref{lem:decomposition}, we compute in $O(|B|^2)$ a decomposition $\L_{r+1}$ of length $k+1$. Hence the running time is $O(|B|^3)$. It remains to prove that $\L_{|B|}$ is basic. If $\L_r=\L_{r+1}$ for some $r<|B|$ then $\L_r$  is basic and $\L_s=\L_r$ for all $r<s\leq |B|$ thus $\L_{|B|}$ is basic. Otherwise, $\L_1,\dots,\L_{|B|}$ are all distinct thus the length of $\L_{|B|}$ is $|B|$. If $\L_{|B|}$ is not basic, by Lemma~\ref{lem:decomposition}, $(A,B)$ has a decomposition of length $|B|+1$, a contradiction. Thus $\L_{|B|}$ is basic.
\end{proof}

\begin{example}
Combining Ex. \ref{ex:split} and \ref{ex:reduction}, we get that the basic decomposition of 
\[
(C_6,3C_6+8C12)
\]
is 
\[
(2C_3,8C_3,4),(C_6,3C_6,1)
\]
and so 
\[
\Sol(C_6,3C_6+8C12)=(\Sol(2C_3,8C_3)\m 4)+\Sol(C_6,3C_6).
\] 
This can be easily checked from Ex. \ref{ex:split} and \ref{ex:reduction}. 
\end{example}

%%%%%%%%%%%%%%%%%%%%%%%%%%%%%%%%%%%%%%%%%%%%%%%%%%%%%%%%%%%%%%%%%%%%%%%%%%%%%%%%%%%%%%%%%
%%%%%%%%%%%%%%%%%%%%%%%%%%%%%%%%%%%%%%%%%%%%%%%%%%%%%%%%%%%%%%%%%%%%%%%%%%%%%%%%%%%%%%%%%
\section{Principal support}\label{sec:principal_support}
%%%%%%%%%%%%%%%%%%%%%%%%%%%%%%%%%%%%%%%%%%%%%%%%%%%%%%%%%%%%%%%%%%%%%%%%%%%%%%%%%%%%%%%%%
%%%%%%%%%%%%%%%%%%%%%%%%%%%%%%%%%%%%%%%%%%%%%%%%%%%%%%%%%%%%%%%%%%%%%%%%%%%%%%%%%%%%%%%%%

The \EM{principal support} of an instance $(A,B)$ is 
\[
\EM{L'(A,B)}=L(A,B)\setminus L(B).
\]
If $X$ is a solution of $(A,B)$, we have $L(X)\subseteq L(A,B)$. But actually, all the information on $X$ is contained on the principal support: if $X'$ is a solution of $(A,B)$ with $X'(x)=X(x)$ for all $x\in L'(A,B)$, then $X'=X$. This is a consequence of the following lemma. 

\begin{lemma}\label{lem:principal_support}
Let $(A,B)$ be an instance, and let $X$ be a sum of cycles with $L(X)\subseteq L'(A,B)$. There exists at most one sum of cycles $Y$ with $L(Y)\subseteq L(B)$ such that $A(X+Y)=B$, and there exists an algorithm, with running time in $O(|B|^3)$, that decides if $Y$ exists, and outputs it if it exists. 
\end{lemma}

\begin{proof}
If $AX$ is not contained in $B$, there is certainly no solution containing $X$, and if $AX=B$ we must take $Y=\emptyset$. So suppose that $AX$ is strictly contained in $B$. Let us compute a sequence $(B_0,Y_0,b_0),(B_1,Y_1,b_1),\dots,(B_k,Y_k,b_k)$ as follows. First we set $B_0=B-AX$ and  $Y_0=\emptyset$ and $b_0=0$. Then, $(B_i,Y_i,b_i)$ having been already computed, we do the following:
\begin{itemize}
\item If $B_i=\emptyset$, then we stop the algorithm. 
\item Let $b$ be the minimum of $L(B_i)$.
\item If $AC_b$ is not contained in $B_i$, then we stop the algorithm. 
\item Otherwise, we set $B_{i+1}=B_i-AC_b$ and $Y_{i+1}=Y_i+C_b$ and $b_{i+1}=b$. 
\end{itemize}
We output $Y_k$ if $B_k=\emptyset$, and ``there is no $Y$ with $L(Y)\subseteq L(B)$ such that $A(X+Y)=B$'' otherwise. Since each step can be done in $O(|B|^2)$, and since $k\leq |B|$, the total running time is in $O(|B|^3)$. It is easy to check, by induction on $i$, that 
\begin{equation}\label{eq:B_{i+1}}
B_{i+1}=B_0-AY_{i+1}.
\end{equation}
So $B_k=\emptyset$ is equivalent to $AY_k=B_0$, which is equivalent to $A(X+Y_k)=B$. Hence it remains to prove the uniqueness of $Y_k$. So let $Y$ be a sum of cycles with $L(Y)\subseteq L(B)$ such that $A(X+Y)=B$, which is equivalent to $AY=B_0$, and let us prove that $Y=Y_k$. It is sufficient to prove that $Y_i\subseteq Y$ for all $0\leq i\leq k$, and that a strict inclusion implies $i<k$. The base case $Y_0\subseteq Y$ is obvious. Suppose that $Y_i\subseteq Y$, with $0\leq i\leq k$. By \eqref{eq:B_{i+1}} we have 
\[
B_i=B_0-AY_i=AY-AY_i=A(Y-Y_i).
\]
If $B_i=\emptyset$ then $i=k$ and $Y_k=Y_i=Y$. Otherwise, since $B_i=A(Y-Y_i)$, the minimum $b$ of $L(B_i)$ has a divisor $y\in L(Y-Y_i)$. Since $y\in L(B)$ and $(A,B)$ has a solution, $y$ has a divisor $a\in L(A)$. Then $y=a\lor y\in L(B_i)$, and since $y\mid b$ we have $y=b$ by the choice of $b$. Hence $b\in L(Y-Y_i)$ and since $A(Y-Y_i)=B_i$ we have $AC_b\subseteq B_i$. Consequently, $i<k$ and $Y_{i+1}=Y_i+C_b\subseteq  Y$, completing the induction.
\end{proof}

Consequently, to compute all the solutions of an instance $(A,B)$, it is enough to enumerate the sum of cycles $X$ with $L(X)\subseteq L(A,B)\setminus L(B)$ with at most $|B|/|A|$ vertices, and to check if $X$ can be extended to a global solution with the algorithm of Lemma~\ref{lem:principal_support}. This gives the following refinement of the brute force approach on the support.  

\begin{quote}
{\bf Brute force approach on the principal support.} If $n=|B|/|A|$ is not an integer, return $\emptyset$. Otherwise, compute $L'(A,B)$. For all $1\leq m\leq n$, enumerate all the partitions of $m$ with parts in $L'(A,B)$, take the corresponding sum of cycles $X$, select those for which there exists a sum of cycles $Y$ with $Y\subseteq L(B)$ such that $X+Y$ is a solution, and return the selected sums $X+Y$. 
\end{quote}
Since $L(A,B)$ can be computed in $O(|B|^2)$, $L'(A,B)$ can be computed with the same complexity. Next, to enumerate the partitions of $1\leq m\leq n$ with parts in $L=L'(A,B)$, we proceed almost exactly as in Section~\ref{sec:support}. In such a partition, the number of occurrences of $x\in L$ is between $0$ and $\lfloor n/x\rfloor$. Hence we can enumerate all the functions $f:L\to\N$ such that $f(x)\leq 1+\lfloor n/x\rfloor$ for all $x\in L$, and select those which correspond to partitions such that $1\leq \sum_{x\in L}xf(x)\leq n$. It is easy to show, by induction on $|L|$, that the number of such functions is at most $n^{|L|-1}$. As explained in Section~\ref{sec:support}, this enumeration can be done in $O(n^{|L|})$. Then, for each partition we take the corresponding sums of cycles $X$, and we use the algorithm of Lemma \ref{lem:principal_support} to check in $O(|B|^3)$ if there exists a sum of cycles $Y$ with $Y\subseteq L(B)$ such that $X+Y$ is a solution. The algorithm is correct thanks to this lemma, and the total complexity is 
\begin{equation*}
O\left(|B|^3\left(\frac{|B|}{|A|}\right)^{|L(A,B)\setminus L(B)|}\right).
\end{equation*}

\medskip
The interesting fact is that the size of the principal support is sometimes significantly smaller than the support, and may lead to polynomial algorithms for some particular classes of instances. We illustrate this with two simple examples.  

\medskip
Suppose first that $(A,B)$ is a basic instance and that $L(A)\setminus\{1\}$ only contains primes. Then the set of solutions can be computed in $O(|B|^4)$, and there are at most $|B|/|A|$ solutions. Indeed, let $x\in L(A,B)$, $x>1$. Let $p$ be a prime that divides $x$. Since the instance is basic, $x$ divides $\lcm L(A)$, which is the product of the primes contained in $L(A)$. We deduce that $p\in L(A)$, and thus $x=p\lor x\in L(B)$. This proves that $L(A,B)\setminus L(B)\subseteq \{1\}$ and thus the brute force approach on the principal support runs in $O(|B|^4)$, and the number of solutions is at most $|B|/|A|$. 

\medskip
Unfortunately, we were not able to find a polynomial algorithm to decide if $A\mid B$ when $L(A)$ is a primitive set of prime powers, that is, $L(A)=\{p_1^{\alpha_1},\dots,p_k^{\alpha_k}\}$ for distinct primes $p_1,\dots,p_k$.    

\medskip
Suppose now that $(A,B)$ is a consistent instance and that $1\in L(B)$. By the easy lemma below, $L'(A,B)=\emptyset$, hence there is a unique solution which can be computed in $O(|B|^3)$ with the brute force approach on the principal support.

\begin{lemma} 
If $(A,B)$ is a consistent instance, then 
\[
L'(A,B)=L(A,B)\setminus \Mult(L(B)).
\]
\end{lemma}

\begin{proof}
Suppose that there exists $x\in L(A,B)$ and $b\in L(B)$ with $b\mid x$. Since the instance is consistent, there exists $a\in L(A)$ such that $a\mid b$, and again because the instance is consistent we have $a\lor x\in L(B)$. But since $a\mid x$, we have $a\lor x=x$ and thus $x\in L(B)$. This proves that $L'(A,B)\subseteq L(A,B)\setminus \Mult(L(B))$, and the converse inclusion is obvious since $L(B)\subseteq \Mult(L(B))$. 
\end{proof}

%%%%%%%%%%%%%%%%%%%%%%%%%%%%%%%%%%%%%%%%%%%%%%%%%%%%%%%%%%%%%%%%%%%%%%%%%%%%%%%%%%%%%%%%%
%%%%%%%%%%%%%%%%%%%%%%%%%%%%%%%%%%%%%%%%%%%%%%%%%%%%%%%%%%%%%%%%%%%%%%%%%%%%%%%%%%%%%%%%%
\section{Concluding remarks}\label{sec:conclusion}
%%%%%%%%%%%%%%%%%%%%%%%%%%%%%%%%%%%%%%%%%%%%%%%%%%%%%%%%%%%%%%%%%%%%%%%%%%%%%%%%%%%%%%%%%
%%%%%%%%%%%%%%%%%%%%%%%%%%%%%%%%%%%%%%%%%%%%%%%%%%%%%%%%%%%%%%%%%%%%%%%%%%%%%%%%%%%%%%%%%

\begin{itemize}
\item 
Our main contribution is to show that, up to polynomial transformations, we can restrict the division problem for sums of cycles to basic compact and consistent instances $(A,B)$ (Lemma~\ref{lem:decompo2}). By the basic property, the exponent term $|L(A,B)|$ in the complexity of the brute force approach on the support is bounded as a function of $A$ only, namely $|L(A,B)|\leq \div(\lcm L(A))$, leading to a polynomial algorithm when $A$ is fixed. But this bound does not show that the complexity is significantly better than the (very naïve) brute force approach, running in $|B|e^{O(\sqrt{n})}$ where $n=|B|/|A|$. The problem is that it is rather difficult to bound $|L(A,B)|$ according to $n$, and a real improvement would be obtained by showing that $|L(A,B)|=o(\sqrt{n})$. Actually, considering instead the brute force approach on the principal support, showing $|L'(A,B)|=o(\sqrt{n})$ would be sufficient. Anyway, the existence of an algorithm for the division problem for sums of cycles running in $\mathrm{poly}(|B|)e^{o(\sqrt{n})}$ remains an open problem. 
\item
One way to better understand what makes the problem apparently difficult is to find polynomial algorithms for particular classes of instances. For instance, it is not difficult to prove that the problem is polynomial when $L(B)$ is a \EM{chain}, that is, when $b\mid b'$ or $b'\mid b$ for all $b,b'\in L(B)$. This leads to consider the dual situation, where $L(B)$ is an \EM{antichain} (called primitive set in number theory), that is, there is no distinct $b,b'\in L(B)$ such that $b\mid b'$. This case seems much more challenging. In that direction, we show that there is a polynomial algorithm when $L(A)$ is a set of primes (Section~\ref{sec:principal_support}), but, unfortunately, we were not able to prove the same property when $L(A)$ is an antichain of prime powers. This case is particularly interesting since it is a typical situation where the bound $|L(A,B)|\leq \div(\lcm L(A))$ gives few information on the real size of $L(A,B)$.
\item 
A principal drawback of our method is that, before using a brute force approach, the only information used to analyze an instance $(A,B)$ is the support of $A$ and $B$ and the size $n=|B|/|A|$ of potential solutions: we do not consider the multiplicities of cycles in $A$ or $B$, that is, the quantities $A(a)$ and $B(b)$ for $a\in L(A)$ and $b\in L(B)$. It is very likely that these quantities will have to be analyzed in order to make progress. To illustrate this, suppose that $(A,B)$ is a consistent instance and that all the cycles in $B$ have distinct lengths, that is, $B(b)=1$ for all $b\in L(B)$. One easily checks that the decomposition preserves this property, thus we can also suppose that $(A,B)$ is basic. Then the problem is very simple: either $C_1$ is the unique solution, or there is no solution. Indeed, suppose that $AX=B$. Let $x\in L(X)$, $a\in L(A)$ and $b=a\lor x$. If $a$ and $x$ are not coprime, that is $ax>b$, then, by~\eqref{eq:mult}, we have $B(b)\geq \frac{1}{b} aX(a)xX(x)>1$, a contradiction. Thus any $x\in L(X)$ is coprime with any $a\in L(A)$. Suppose now that there exists $x\in L(X)$ with $x>1$ and let $p$ be a prime that divides $x$. Then $L(B)$ contains a multiple of $p$, and since the instance is basic, we deduce that $p$ divides some $a\in L(A)$. But then $a$ and $x$ are not coprime, a contradiction. We deduce that $X=nC_1$ with $n=|B|/|A|$. But if $n>1$ then, for any $a\in L(A)$, we have $B(a)\geq \frac{1}{a} aX(a)X(1)=X(a)X(1)=X(a)n>1$, a contradiction. Thus $X=C_1$. 
\item
Beyond sums of cycles, that there is a polynomial algorithm for the division problem inside the whole set of functional digraphs seems a challenging problem. In the spirit of this paper, it would be at least interesting to prove that, for every fixed functional digraph $A$, there is a polynomial algorithm to decide, given any functional digraph $B$, if $A$ divides $B$. As a small step in this direction, using \cite[Corollary 14]{naquin2024factorisation}, it seems rather easy to adapt our arguments to prove that, for every sum of cycles $A$ and every functional digraph $B$, there is an algorithm to decide if $A$ divides $B$ which is polynomial when $A$ is fixed and with a complexity very similar to that given in \eqref{eq:total_run}.  
\item
We proved that, for every $\epsilon>0$ and infinitely many $n$, there exists an instance $(A,B)$ with $|A|=n$ and $|B|=n^2$ such that the number of (irreducible) solutions of $(A,B)$ and the number of irreducible factorizations of $B$ is at least \eqref{eq:super_poly}, hence super-polynomial in $n$. One may ask whether these lower bounds are far from the truth, and thus ask for non-trivial upper bounds. The only upper bound we give concerns the number of irreducible factorizations of a sum of cycles $X$ with $n$ vertices, which is at most $ne^{c_0\sqrt{n}}$. But this bound is certainly very loose since it is actually a bound on $p^\times(n)$, the number of ways of producing sums of cycles with $n$ vertices from products of sums of cycles, each distinct from $C_1$ (the order of the factors in the product being irrelevant). It therefore seems that much more precise upper bounds could be obtained. Note that, from the proof that $p^\times (n)\leq ne^{c_0\sqrt{n}}$ (Appendix \ref{an:product}), we easily deduce that we can enumerate in $e^{O({\sqrt{n}})}$ the products of sums of cycles, each distinct from $C_1$, resulting in sums of cycles with $n$ vertices. Consequently, the irreducibility decision problem for sums of cycles is sub-exponential. The arguments being very rough, huge progress seems possible for this decision problem, which is perhaps even more natural than the division problem. 
\end{itemize}

\paragraph{Acknowledgments} 
The first and last authors were partially supported by the project ANR-24-CE48-7504
``ALARICE'' and the HORIZON-MSCA-2022-SE-01 project ``ACANCOS''. The third author was partially supported by the project NCVCC of VAST.

\appendix 

%%%%%%%%%%%%%%%%%%%%%%%%%%%%%%%%%%%%%%%%%%%%%%%%%%%%%%%%%%%%%%%%%%%%%%%%%%%%%%%%%%%%%%%%%%%%%%%%%%%%%%%
\section{Irreducible and prime sums of cycles}\label{an:irr_prime}
%%%%%%%%%%%%%%%%%%%%%%%%%%%%%%%%%%%%%%%%%%%%%%%%%%%%%%%%%%%%%%%%%%%%%%%%%%%%%%%%%%%%%%%%%%%%%%%%%%%%%%%

Recall that a sum of cycles $X$ is \EM{irreducible} if, for all sums of cycles $A,B$, $X=AB$ implies $A=C_1$ or $B=C_1$. For instance, it is an easy exercise to prove that a single cycle $C_n$ is irreducible if and only if $n$ is a prime power, and this property will be used many times in the following. Let \EM{$\irred(n)$} be the number of irreducible sums of cycles with $n$ vertices, and let $\EM{\red(n)}=p(n)-\irred(n)$ be the number of reducible sums of cycles with $n$ vertices. We will prove that almost all sums of cycles are irreducible using a basic counting argument.

\begin{proposition}
$\irred(n)/p(n)\to 1$ as $n\to\infty$.
\end{proposition}

\begin{proof}
 Let $X$ be a reducible sum of cycles with $n$ vertices. Then there exists two sums of cycles $A,B\neq C_1$ such that $AB=X$. Let $d=|A|$ and suppose, without loss, that $|A|\leq |B|$. Then $2\leq d\leq\sqrt{n}$. Consequently, 
\[
\red(n) \leq
\sum_{\substack{d\mid n \\2\leq d\leq \sqrt{n}}} p(d)p(n/d) \leq
\sum_{\substack{d\mid n \\2\leq d\leq \sqrt{n}}} p(\lceil \sqrt{n}\rceil)p(\lceil n/2\rceil) \leq 
\sqrt{n} p(\lceil \sqrt{n}\rceil)p(\lceil n/2\rceil).
\]
Let $\delta=\sqrt{1/2}$ and $0<\epsilon<\frac{1-\delta}{1+\delta}c_0$. Using \eqref{eq:Erdos_bounds}, for $n$ large enough we have 
\begin{eqnarray*}
\sqrt{n} p(\lceil \sqrt{n}\rceil)p(\lceil n/2\rceil) &\leq &  
\sqrt{n} e^{c_0\left(\sqrt{\lceil \sqrt{n}\rceil}+\sqrt{\lceil n/2\rceil}\right)} \leq 
e^{(c_0+\epsilon)\sqrt{n/2}}=e^{(c_0+\epsilon)\delta\sqrt{n}}.
\end{eqnarray*}
Using again \eqref{eq:Erdos_bounds}, we deduce that, for $n$ large enough,  
\[
\frac{\red(n)}{p(n)}\leq \frac{e^{(c_0+\epsilon)\delta\sqrt{n}}}{e^{(c_0-\epsilon)\sqrt{n}}} =
\frac{1}{e^{\left((c_0-\epsilon)-\delta(c_0+\epsilon)\right)\sqrt{n}}}
\]
By the choice of $\epsilon$, we have $c_0-\epsilon>(c_0+\epsilon)\delta$ so $\red(n)/p(n)\to 0$ as $n\to\infty$, and we are done.
\end{proof}

Recall that a functional digraph $X$ is \EM{prime} if $X\neq C_1$ and, for all functional digraphs $A,B$, $X\mid AB$ implies $X\mid A$ or $X\mid B$. While proving that the semiring of functional digraphs has no prime element is difficult \cite{seifert1971prime}, proving that the same is true for the semiring of sums of cycles is an easy exercise. 

\begin{proposition}
There is no prime element in the semiring of sums of cycles: there is no sum of cycles $X\neq C_1$ such that, for all sums of cycles $A,B$, $X\mid AB$ implies $X\mid A$ or $X\mid B$. 
\end{proposition}

\begin{proof}
There are four cases:
\begin{itemize}
\item Suppose that $X=xC_1$ for some integer $x\geq 2$. If $x$ is not prime then $X$ is reducible, and thus $X$ is not prime. So suppose that $X=pC_1$ for some prime $p$. Then
\[
pC_1\cdot C_p=pC_p=C_p\cdot C_p.
\]
Since $C_p$ is irreducible, we deduce that $X$ is not prime.   
\item
Suppose that $X=C_{p^\alpha}$ for some prime $p$ and $\alpha\geq 1$. Then 
\[
C_{p^\alpha}\cdot pC_{p^{\alpha+1}}=p^{\alpha+1}C_{p^{\alpha+1}}=C_{p^{\alpha+1}}\cdot C_{p^{\alpha+1}}.
\] 
Since $C_{p^{\alpha+1}}$ is irreducible, we deduce that $X$ is not prime. 
\item
Suppose now that $\lcm L(X)$ is a prime power $p^\alpha$ and $|X|>p^{\alpha}$. Then 
\[
X\cdot C_{p^\alpha}=|X|C_{p^\alpha}=C_{p^\alpha}\cdot |X|C_1. 
\] 
Since $|X|>p^\alpha$ and $X$ does not divide $|X|C_1$, we deduce that $X$ is not prime. 
\item
In the other cases, $\ell=\lcm L(X)>1$ and is not a prime power. Let $\ell=p_1^{\alpha_1}p_2^{\alpha_2}\cdots p_k^{\alpha_k}$ be the prime decomposition of $\ell$; by hypothesis $k\geq 2$. We have 
\[
X\cdot C_\ell=|X|C_\ell=|X|C_{p_1^{\alpha_1}}\cdot C_{\ell/p_1^{\alpha_1}}. 
\]
By \eqref{eq:divLAX}, $X$ does not divide $|X|C_{p_1^{\alpha_1}}$ since $L(X)$ contains a multiple of $p_2^{\alpha_2}$, and it does not divide $C_{\ell/p_1^{\alpha_1}}$ since $L(X)$ contains a multiple of $p_1^{\alpha_1}$.
\end{itemize}
\end{proof}

%%%%%%%%%%%%%%%%%%%%%%%%%%%%%%%%%%%%%%%%%%%%%%%%%%%%%%%%%%%%%%%%%%%%%%%%%%%%%%%%%%%%%%%%%%%%%%%%%%%%%%%
\section{Instance with many irreducible solutions}\label{an:many_irr} 
%%%%%%%%%%%%%%%%%%%%%%%%%%%%%%%%%%%%%%%%%%%%%%%%%%%%%%%%%%%%%%%%%%%%%%%%%%%%%%%%%%%%%%%%%%%%%%%%%%%%%%%

Let \EM{$f(n)$} be the number of partitions of $n$ with parts in $\Div(n)$. In this section, we show that the instance $(C_n,nC_n)$ has exactly $f(n)$ solutions (Lemma \ref{lem:solC_nnC_n}). Hopefully, Bowman, Erd\H{o}s and Odlyzko \cite{bowman19926640} obtained an accurate estimate of $\ln f(n)$ depending on $n$ and the number of divisors of $n$: for every $\epsilon>0$ and $n$ large enough, 
\begin{equation}\label{eq:BEO}
(1+\epsilon)\left(\frac{\div(n)}{2}-1\right)\ln n\leq \ln f(n)\leq (1+\epsilon)\frac{\div(n)}{2}\ln n.
\end{equation}
By choosing $n$ with many divisors, we deduce from this that $f$ grows faster than any polynomial: for every $d$, it exists infinitely many $n$ such that $f(n)> n^d$ (Lemma \ref{lem:f(n)}). The number of solutions to the instance $(C_n,nC_n)$ is, in that sense, super-polynomial. Actually, we show something stronger: $(C_n,nC_n)$ has a super-polynomial number of irreducible solutions (Theorem \ref{thm:many_irreducibles}) and, as an immediate consequence, $nC_n$ has a super-polynomial number of irreducible factorizations (Corollary \ref{cor:many_factorization}). 

\begin{lemma}\label{lem:solC_nnC_n}
For all $n\geq 1$, we have $X\in\Sol(C_n,nCn)$ iff $|X|=n$ and $L(X)\subseteq \Div(n)$, so that 
\[
\sol(C_n,nC_n)=f(n). 
\] 
\end{lemma}

\begin{proof}
Suppose that $X$ is a sum of cycles with $n$ vertices and $L(X)\subseteq \Div(n)$. Since $C_dC_n=dC_n$ whenever $d\mid n$, we have $C_n\cdot  X =|X|C_n=n C_n$. Conversely, if $C_n\cdot X=nC_n$ then $X$ has $n$ vertices and, by $\eqref{eq:divLAX}$, we have $L(X)\subseteq \Div(L(nC_n))=\Div(n)$. 
\end{proof}

We now show that $f$ is super-polynomial. Let us defined the \EM{$k$th primorial number} as the product of the first $k$ primes. If $n$ is the $k$th primorial number, then it has $2^k$ divisors and it is well known (see e.g. \cite{crstici2006handbook}) that, for every $\epsilon>0$ and $k$ large enough, 
\begin{equation}\label{eq:primorial}
k\ln k\leq \ln n \leq (1+\epsilon)k\ln k.
\end{equation}

\begin{lemma}\label{lem:f(n)}
For every $\epsilon>0$, if $n$ is a sufficiently large primorial number, then 
\[
f(n)\geq n^{2^{\frac{\ln n}{(1+\epsilon)\ln \ln n}}}.
\]
\end{lemma}

\begin{proof}
Let $0<\delta<\epsilon$, and let $n$ be $k$th primorial number. By \eqref{eq:primorial}, for $k$ large enough  we have 
\[
k\geq \frac{\ln n}{(1+\delta)\ln k}\geq  \frac{\ln n}{(1+\delta)\ln \ln n}.
\]
Using \eqref{eq:BEO}, we obtain, for $k$ large enough,  
\[
\frac{\ln f(n)}{\ln n}
\geq 2^{k-1}
\geq 2^{\frac{\ln n}{(1+\delta)\ln \ln n}-1}
\geq 2^{\frac{\ln n}{(1+\epsilon)\ln \ln n}}.
\]
\end{proof}

\begin{theorem}\label{thm:many_irreducibles}
For every $\epsilon>0$, if $n$ is a sufficiently large primorial number, then the number of irreducible solutions to the instance $(C_n,nC_n)$ is at least 
\[
n^{2^{\frac{\ln n}{(1+\epsilon)\ln \ln n}}}.
\]
\end{theorem}

\begin{proof}
Let $n=p_1p_2\cdots p_k$ be the $k$th primorial number. Let $I$ be the set of sums of cycles $X$ with $n$ vertices such that $L(X)\subseteq \Div(n)$ and such that $X$ has a unique cycle whose length is a multiple of $p_k$, and this cycle is precisely $C_{p_k}$. In other words, $X=C_{p_k}+Y$ for some sum of cycles $Y$ with $n-p_k$ vertices satisfying $L(Y)\subseteq \Div(n/p_k)$. By Lemma \ref{lem:solC_nnC_n}, 
\[
I\subseteq \Sol(C_n,nC_n).
\] 

\medskip
Let $X\in I$, and let us prove that it is irreducible. Let $A,B$ be sums of cycles such that $X=AB$. Since $X$ contains $C_{p_k}$, which is irreducible, at least one of $A,B$ contains this cycle, say $B$ without loss. If $A$ contains $C_{p_k\ell}$ for some $\ell\geq 1$, then $X=AB$ contains $C_{p_k\ell}C_{p_k}=p_kC_{p_k\ell}$, so $X$ contains at least $p_k$ cycles whose length is a multiple of $p_k$, a contradiction. If $A$ contains $C_{p_i\ell}$ for some $1\leq i<k$ and $\ell\geq 1$, then $X=AB$ contains $C_{p_i\ell}C_{p_k}=C_{p_ip_k\ell}$ and this contradicts the fact that $C_{p_k}$ is the unique cycle of $X$ whose length is a multiple of $p_k$. Hence $L(A)$ has no multiple of $p_i$ for $1\leq i\leq k$. By \eqref{eq:divLAX}, we have $L(A)\subseteq \Div(L(X))\subseteq \Div(n)$, and we deduce that $L(A)=\{1\}$. So $A=aC_1$ for some $a\geq 1$. If $a\geq 2$ then $X=AB$ contains $aC_1C_{p_k}=aC_{p_k}$ and this contradicts the fact that $X$ has a unique cycle of length $p_k$. Consequently, $A=C_1$ and thus $X$ is irreducible.  

\medskip
We now give a lower bound on $|I|$, which is exactly the number of partitions of $n-p_k$ with parts in $\Div(n/p_k)$. For $k\geq 2$ we have $n-p_k\geq n/p_k$ and thus
\[
|I|= p_{\Div(n/p_k)}(n-p_k)\geq p_{\Div(n/p_k)}(n/p_k)=f(n/p_k).
\] 
Let $0<\delta<\epsilon$. Since $p_k\leq  2 k\ln k$ (see e.g. \cite{rosser1962approximate}), using \eqref{eq:primorial} we get $p_k\leq 2\ln n$ and thus $\ln p_k\leq \ln \ln n+1$. Since $n/p_k$ is the $(k-1)$th primorial number, we deduce from Lemma~\ref{lem:f(n)} that, for $k$ large enough, 
\[
\ln |I|\geq \ln f(n/p_k)\geq 2^{\frac{\ln n-\ln\ln n-1}{(1+\delta)\ln \ln n}}(\ln n-\ln\ln n-1)\geq 2^{\frac{\ln n}{(1+\epsilon)\ln \ln n}}\ln n.
\]
\end{proof}

The number of \EM{irreducible factorizations} of a sum of cycles $X$ is the number of ways of expressing $X$ as a product of irreducible sums of cycles, the order of the factors in the product being irrelevant.

\begin{corollary}\label{cor:many_factorization}
For every $\epsilon>0$, if $n$ is a sufficiently large primorial number, then the number of irreducible factorizations of $nC_n$ is at least 
\[
n^{2^{\frac{\ln n}{(1+\epsilon)\ln \ln n}}}.
\]
\end{corollary}

\begin{proof}
Let $n=p_1p_2\cdots p_k$ be the $k$th primorial number. Then $C_n$ has a unique irreducible factorization, which is $C_{p_1} C_{p_2}\cdots C_{p_k}$. Let $X$ be an irreducible sum of cycles such that $C_n\cdot X=nC_n$. Then $C_{p_1}C_{p_2}\cdots C_{p_k}\cdot X$ is an irreducible factorization of $nC_n$, and we are done using Theorem~\ref{thm:many_irreducibles}. 
\end{proof}

%%%%%%%%%%%%%%%%%%%%%%%%%%%%%%%%%%%%%%%%%%%%%%%%%%%%%%%%%%%%%%%%%%%%%%%%%%%%%%%%%%%%%%%%%%%%%%%%%%%%%%%
\section{Products of sums of cycles resulting in $n$ vertices}\label{an:product}
%%%%%%%%%%%%%%%%%%%%%%%%%%%%%%%%%%%%%%%%%%%%%%%%%%%%%%%%%%%%%%%%%%%%%%%%%%%%%%%%%%%%%%%%%%%%%%%%%%%%%%%

For $n\geq 2$, let \EM{$P^*(n)$} be the set of sequences of integers $(a_1,a_2,\dots,a_k)$ such that $2\leq a_1\leq a_2\leq \cdots \leq a_k$ and $a_1a_2\cdots a_k=n$. Thus $\EM{p^*(n)}=|P^*(n)|$ is the number of ways of expressing $n$ as the product of positive integers, each distinct from $1$, the order of the factors in the product being irrelevant. Dodd and Mattics \cite{DM87} proved that 
\begin{equation}\label{eq:p*n}
p^*(n)\leq n.
\end{equation}

\medskip
We will consider the same concept for sums of cycles. For $n\geq 2$, let \EM{$P^{\small\times}(n)$} be the set of sequences of sums of cycles $(A_1,A_2,\dots,A_k)$ such that $2\leq |A_1|\leq |A_2|\leq \cdots \leq |A_k|$ and $|A_1A_2\cdots A_k|=n$. Thus $\EM{p^{\small\times}(n)}=|P^{\small\times}(n)|$ is the number of ways of producing sums of cycles with $n$ vertices from products of sums of cycles, each distinct from $C_1$, the order of the factors in the product being irrelevant.

\begin{proposition}
For all $n\geq 2$, 
\[
p^{\small\times}(n)\leq ne^{c_0\sqrt{n}}. 
\]
\end{proposition}

\begin{proof}
Let $(A_1,A_2,\dots,A_k)$ be a sequence in $P^{\small\times}(n)$. Let us call $(|A_1|,\dots,|A_k|)$ the profile sequence; it belongs to $P^*(n)$. Let $(a_1,\dots,a_k)$ be a sequence in $P^*(n)$, and let $\sigma$ be the number of sequences in $P^{\small\times}(n)$ with this profile. Hence $\sigma=p(a_1)p(a_2)\cdots p(a_k)$. We will prove that $\sigma \leq e^{c_0\sqrt{n}}$. Since, by \eqref{eq:p*n}, there are at most $n$ possible profiles, this proves the theorem. 

\medskip
Using \eqref{eq:Erdos_bounds}, we have 
\[
\sigma\leq e^{c_0\sum_{i=1}^n \sqrt{a_i}},
\]
If $a_1\geq 4$ then 
\[
\sigma\leq e^{c_0\sum_{i=1}^n \sqrt{a_i}} \leq e^{c_0\prod_{i=1}^n\sqrt{a_i}}=e^{c_0\sqrt{n}}.
\]
If $a_k\leq 3$ then, since $p(2)=2$ and $p(3)=3$, we have 
\[
\sigma=a_1a_2\cdots a_k=n\leq e^{c_0\sqrt{n}}.
\]
Otherwise, there is $1\leq \ell<k$ with $a_\ell<4\leq a_{\ell+1}$. Let $n_1=a_1a_2\dots a_\ell$ and $n_2=n/n_1$. With the previous arguments, we get 
\[
\sigma\leq n_1\cdot e^{c_0\sqrt{n_2}}.
\]
If $n_1\leq 3$ we easily check that $\sigma\leq e^{c_0\sqrt{n}}$, and if $n_1\geq 4$ then 
\[
\sigma\leq n_1\cdot e^{c_0\sqrt{n_2}}\leq e^{\sqrt{n_1}}e^{c_0\sqrt{n_2}}\leq e^{c_0(\sqrt{n_1}+\sqrt{n_2})}\leq e^{c_0\sqrt{n}}.
\]
Thus $\sigma\leq e^{c_0\sqrt{n}}$ in every case.
\end{proof}

Let $X$ be a sum of cycles with $n$ vertices, and let \EM{$\fact(X)$} be the number of irreducible factorizations of $X$. We obviously have 
\[
\fact(X)\leq p^\times(n)\leq ne^{c_0\sqrt{n}}.
\]
This proves, with very rough arguments, the assertion given in the introduction: the number of irreducible factorizations of a sum of cycles with $n$ vertices is at most $e^{O(\sqrt{n})}$.

%%%%%%%%%%%%%%%%%%%%%%%%%
\bibliographystyle{plain}
\bibliography{BIB}

@inproceedings{bridoux2024dividing,
  title={Dividing permutations in the semiring of functional digraphs},
  author={Bridoux, Florian and Crespelle, Christophe and Phan, Thi Ha Duong and Richard, Adrien},
  booktitle={International Workshop on Cellular Automata and Discrete Complex Systems},
  pages={95--107},
  year={2024},
  organization={Springer}
}

@article{defrain2024polynomial,
  title={Polynomial-delay generation of functional digraphs up to isomorphism},
  author={Defrain, Oscar and Porreca, Antonio E and Timofeeva, Ekaterina},
  journal={Discrete Applied Mathematics},
  volume={357},
  pages={24--33},
  year={2024},
  publisher={Elsevier}
}

@article{erdos1942elementary,
  title={On an elementary proof of some asymptotic formulas in the theory of partitions},
  author={Erd{\"o}s, Paul},
  journal={Annals of Mathematics},
  volume={43},
  number={3},
  pages={437--450},
  year={1942},
  publisher={JSTOR}
}

@article{rosser1962approximate,
  title={Approximate formulas for some functions of prime numbers},
  author={Rosser, J Barkley and Schoenfeld, Lowell},
  journal={Illinois Journal of Mathematics},
  volume={6},
  number={1},
  pages={64--94},
  year={1962},
  publisher={Duke University Press}
}

@book{crstici2006handbook,
  title={Handbook of Number Theory I},
  author={Crstici, Borislav and Mitrinovic, Dragoslav S and S{\'a}ndor, J{\'o}zsef},
  year={2006},
  publisher={Springer}
}

@mastersthesis{Dorigatti2017,
title={Algorithms and complexity of the algebraic analysis of finite discrete dynamical systems},
author={Valentina Dorigatti},
school={Universit\'a degli Studi di Milano Bicocca},
year={2017}
}

@phdthesis{riva2022factorisation,
  title={Factorisation of discrete dynamical systems},
  author={Riva, Sara},
  year={2022},
  school={Universit{\'e} C{\^o}te d'Azur; Universit{\`a} degli studi di Milano-Bicocca}
}

@article{dore2024decomposition,
  title={Decomposition and factorisation of transients in functional graphs},
  author={Dor{\'e}, Fran{\c{c}}ois and Formenti, Enrico and Porreca, Antonio E and Riva, Sara},
  journal={Theoretical Computer Science},
  volume={999},
  pages={114514},
  year={2024},
  publisher={Elsevier}
}

@inproceedings{dore2024roots,
  title={Roots in the semiring of finite deterministic dynamical systems},
  author={Dor{\'e}, Fran{\c{c}}ois and Perrot, K{\'e}vin and Porreca, Antonio E and Riva, Sara and Rolland, Marius},
  booktitle={International Workshop on Cellular Automata and Discrete Complex Systems},
  pages={120--132},
  year={2024},
  organization={Springer}
}

@inproceedings{dennunzio2019solving,
  title={Solving equations on discrete dynamical systems},
  author={Dennunzio, Alberto and Formenti, Enrico and Margara, Luciano and Montmirail, Valentin and Riva, Sara},
  booktitle={International Meeting on Computational Intelligence Methods for Bioinformatics and Biostatistics},
  pages={119--132},
  year={2019},
  organization={Springer}
}

@inproceedings{dennunzio2018polynomial,
  title={Polynomial equations over finite, discrete-time dynamical systems},
  author={Dennunzio, Alberto and Dorigatti, Valentina and Formenti, Enrico and Manzoni, Luca and Porreca, Antonio E},
  booktitle={Cellular Automata: 13th International Conference on Cellular Automata for Research and Industry, ACRI 2018, Como, Italy, September 17--21, 2018, Proceedings 13},
  pages={298--306},
  year={2018},
  organization={Springer}
}

@article{naquin2024factorisation,
  title={Factorisation in the semiring of finite dynamical systems},
  author={Naquin, {\'E}mile and Gadouleau, Maximilien},
  journal={Theoretical Computer Science},
  volume={998},
  pages={114509},
  year={2024},
  publisher={Elsevier}
}

@article{seifert1971prime,
  title={On prime binary relational structures},
  author={Seifert, Ralph},
  journal={Fundamenta Mathematicae},
  volume={70},
  number={2},
  pages={187--203},
  year={1971},
  publisher={Polska Akademia Nauk. Instytut Matematyczny PAN}
}

@inproceedings{hopcroft1974linear,
  title={Linear time algorithm for isomorphism of planar graphs (preliminary report)},
  author={Hopcroft, John E and Wong, Jin-Kue},
  booktitle={Proceedings of the sixth annual ACM symposium on Theory of computing},
  pages={172--184},
  year={1974}
}

@article{ehrenfeucht2007reaction,
  title={Reaction systems},
  author={Ehrenfeucht, Andrzej and Rozenberg, Grzegorz},
  journal={Fundamenta informaticae},
  volume={75},
  number={1-4},
  pages={263--280},
  year={2007},
  publisher={IOS Press}
}

@article{hardy1918asymptotic,
  title={Asymptotic formula{\ae} in combinatory analysis},
  author={Hardy, Godfrey H and Ramanujan, Srinivasa},
  journal={Proceedings of the London Mathematical Society},
  volume={2},
  number={1},
  pages={75--115},
  year={1918},
  publisher={Wiley Online Library}
}

@article{bowman19926640,
  title={6640},
  author={Bowman, Douglas and Erdos, Paul and Odlyzko, Andrew M},
  journal={The American Mathematical Monthly},
  volume={99},
  number={3},
  pages={276--277},
  year={1992},
  publisher={JSTOR}
}

@book{nathanson2000elementary,
  title={Elementary methods in number theory},
  author={Nathanson, Melvyn B},
  volume={195},
  year={2000},
  publisher={Springer Science \& Business Media}
}

@book{knuth2005art,
  title={The Art of Computer Programming, Volume 4, Fascicle 3: Generating All Combinations and Partitions},
  author={Knuth, Donald E},
  year={2005},
  publisher={Addison-Wesley Professional}
}

@article{dennunzio2024note,
  title={A note on solving basic equations over the semiring of functional digraphs},
  author={Dennunzio, Alberto and Formenti, Enrico and Margara, Luciano and Riva, Sara},
  journal={arXiv preprint arXiv:2402.16923},
  year={2024}
}

@article{DFMR23,
  title={An algorithmic pipeline for solving equations over discrete dynamical systems modelling hypothesis on real phenomena},
  author={Dennunzio, Alberto and Formenti, Enrico and Margara, Luciano and Riva, Sara},
  journal={Journal of Computational Science},
  volume={66},
  pages={101932},
  year={2023},
  publisher={Elsevier}
}

@article{GMP20,
  title={Profiles of dynamical systems and their algebra},
  author={Gaze-Maillot, Caroline and Porreca, Antonio E},
  journal={arXiv preprint arXiv:2008.00843},
  year={2020}
}

@inproceedings{DDFMP18,
  title={Polynomial equations over finite, discrete-time dynamical systems},
  author={Dennunzio, Alberto and Dorigatti, Valentina and Formenti, Enrico and Manzoni, Luca and Porreca, Antonio E},
  booktitle={Cellular Automata: 13th International Conference on Cellular Automata for Research and Industry, ACRI 2018, Como, Italy, September 17--21, 2018, Proceedings 13},
  pages={298--306},
  year={2018},
  organization={Springer}
}

@article{DM87,
  title={Estimating the number of multiplicative partitions},
  author={Dodd, FW and Mattics, LE},
  journal={The Rocky Mountain journal of mathematics},
  volume={17},
  number={4},
  pages={797--813},
  year={1987},
  publisher={JSTOR}
}

@article{G20,
  title={On the influence of the interaction graph on a finite dynamical system},
  author={Gadouleau, Maximilien},
  journal={Natural Computing},
  volume={19},
  number={1},
  pages={15--28},
  year={2020},
  publisher={Springer}
}

@article{R19,
  title={Positive and negative cycles in Boolean networks},
  author={Richard, A.},
  journal={Journal of theoretical biology},
  volume={463},
  pages={67--76},
  year={2019},
  publisher={Elsevier}
}

@inproceedings{FO89,
  title={Random mapping statistics},
  author={Flajolet, Philippe and Odlyzko, Andrew M},
  booktitle={Workshop on the Theory and Application of of Cryptographic Techniques},
  pages={329-354},
  year={1989},
  organization={Springer}
}

@article{GT83,
author = {Goles, E. and Tchuente, M.},
title ={Iterative behaviour of generalized majority functions},
journal ={Mathematical Social Sciences},
volume ={4},
pages = {197-204},
year = {1982}
}

@article{PS83,
author = {Poljak, S. and Sura, M.},
title ={On periodical behaviour in societies with symmetric influences},
journal ={Combinatorica},
volume ={3},
pages = {119-121},
year = {1982}
}

@article{H82,
author = {Hopfield, J.},
title ={Neural networks and physical systems with emergent collective computational abilities},
journal ={Proc. Nat. Acad. Sc. U.S.A.},
volume ={79},
pages = {2554-2558},
year = {1982}
}

@article{G85,
author = {Goles, E.},
title ={Dynamics of positive automata networks},
journal ={Theoretical Computer Science},
volume ={41},
pages = {19-32},
year = {1985}
}

@article{MP43,
author = {Mac Culloch, W.~S. and Pitts, W.~S.},
title ={A logical calculus of the ideas immanent in nervous activity},
journal ={Bull. Math Bio. Phys.},
volume ={5},
pages = {113-115},
year = {1943}
}

@ARTICLE{T73,
  author = {Thomas, R.},
  title = {{B}oolean formalization of genetic control circuits},
  journal = {Journal of Theoretical Biology},
  year = {1973},
  volume = {42},
  pages = {563-585},
  number = {3},
  doi = {10.1016/0022-5193(73)90247-6},
  issn = {0022-5193}
}

@BOOK{TA90,
  title = {Biological Feedback},
  publisher = {CRC Press},
  year = {1990},
  author = {Thomas, R. and d'Ari, R.}
}

@ARTICLE{K69,
  author = {Kauffman, S. A.},
  title = {Metabolic stability and epigenesis in randomly connected nets},
  journal = {Journal of Theoretical Biology},
  year = {1969},
  volume = {22},
  pages = {437-467},
  doi = {10.1016/0022-5193(69)90015-0}
}

@book{GM90,
 author = {Goles, E. and Mart{\'\i}nez, S.},
 title = {Neural and Automata Networks: Dynamical Behavior and Applications},
 year = {1990},
 publisher = {Kluwer Academic Publishers}
}

@ARTICLE{J02,
    author = {Hidde De Jong},
    title = {Modeling and simulation of genetic regulatory systems: A literature review},
    journal = {Journal of Computational Biology},
    year = {2002},
    volume = {9},
    pages = {67-103}
}

@ARTICLE{TK01,
  author = {Thomas, R. and Kaufman, M.},
  title = {Multistationarity, the basis of cell differentiation and memory.
	{II}. {L}ogical analysis of regulatory networks in terms of feedback
	circuits},
  journal = {Chaos: An Interdisciplinary Journal of Nonlinear Science},
  year = {2001},
  volume = {11},
  pages = {180-195},
  number = {1},
  doi = {10.1063/1.1349893},
  publisher = {AIP}
}
%%%%%%%%%%%%%%%%%%%%%%%%%

\end{document}